\documentclass[10pt,onecolumn]{IEEEtran}

\usepackage{graphics} 
\usepackage{epsfig} 

\usepackage{times,verbatim} 
\usepackage{amsmath, amssymb}
\usepackage{amsthm}
\usepackage{thmtools}
\usepackage{graphicx}
\usepackage{subfigure}
\usepackage{epstopdf}
\usepackage{algorithm}
\usepackage[noend]{algorithmic}
\usepackage{enumerate}
\usepackage{color}

\usepackage{standalone}

\usepackage{todonotes}
\usepackage{filecontents}

\newtheorem{theorem}{Theorem}[section]
\newtheorem{lemma}[theorem]{Lemma}

\newtheorem{problem}[theorem]{Problem}
\newtheorem{proposition}[theorem]{Proposition}

\newtheorem{definition}[theorem]{Definition}

\newtheorem{remark}[theorem]{Remark}
\numberwithin{equation}{section}
\newtheorem{defProp}[theorem]{Definition (Proposition)}

\newcommand{\R}{{\mathbb{R}}}

\newcommand{\B}{{\mathbb B}}

\newcommand{\N}{{\mathbb{N}}}

\newcommand{\argmin}{\textrm{arg}\min}

\renewcommand{\O}{\mathcal{O}}
\newcommand{\M}{\mathcal{M}}

\renewcommand{\tilde}[1]{\widetilde{#1}}

\newcommand{\half}{\frac{1}{2}}
\newcommand{\norm}[1]{\left\Vert#1\right\Vert_2}
\renewcommand{\matrix}[1]{\begin{bmatrix}#1\end{bmatrix}}

\newcommand{\supp}{\textrm{supp}}

\usepackage{tikz}
\usetikzlibrary{shapes,shadows,calc,snakes}
\usetikzlibrary{decorations.pathmorphing,patterns}
\usetikzlibrary{decorations}
\usepackage{pgfplots}

\title{\LARGE \bf 
Secure State Estimation For Cyber Physical Systems Under \\Sensor Attacks: A Satisfiability Modulo Theory Approach
}

\author{Yasser Shoukry, Pierluigi Nuzzo,  Alberto Puggelli, 
\\Alberto L. Sangiovanni-Vincentelli, Sanjit A. Seshia, and Paulo Tabuada 
\thanks{Y. Shoukry and P. Tabuada with Electrical
Engineering Department, UCLA, {\tt\footnotesize \{yshoukry, tabuada\}@ucla.edu}}%
\thanks{P. Nuzzo, A. Puggelli,  A. L. Sangiovanni-Vincentelli, and S. A. Seshia are with Electrical Engineering and Computer Science Department, UC Berkeley, {\tt\footnotesize \{puggelli,nuzzo,alberto,sseshia\}@eecs.berkeley.edu}}
\thanks{This work was partially sponsored by the NSF award 1136174, by DARPA under agreement number FA8750-12-2-0247, by TerraSwarm, one of six centers of STARnet, a Semiconductor Research Corporation program sponsored by MARCO and DARPA, and by the NSF project ExCAPE: Expeditions in Computer Augmented Program Engineering (award 1138996). The U.S. Government
is authorized to reproduce and distribute reprints for Governmental purposes
notwithstanding any copyright notation thereon. The views and conclusions
contained herein are those of the authors and should not be interpreted
as necessarily representing the official policies or endorsements, either
expressed or implied, of NSF, DARPA or the U.S. Government.
}
}

\begin{document}

\maketitle

\begin{abstract}
Secure state estimation is the problem of estimating the state of a dynamical system from a set of noisy and adversarially corrupted measurements. The secure state estimation is a combinatorial problem, which has been addressed either by brute force search, suffering from scalability issues, or via convex relaxations using algorithms that can terminate in polynomial time but are not necessarily sound. In this paper, we present a novel algorithm  that uses a Satisfiability-Modulo-Theory approach to lessen the intrinsic combinatorial complexity of the problem. By leveraging results from formal methods over real numbers, we provide guarantees on the soundness and completeness of our algorithm. Moreover, we provide upper bounds on the runtime performance of the proposed algorithm in order to proclaim the scalability of the proposed algorithm. The scalability argument is then supported by numerical simulations showing an order of magnitude decrease in the runtime performance with alternative techniques.  Finally, we demonstrate its application to the problem of controlling an unmanned ground vehicle.
\end{abstract}

\section{Introduction} \label{sec:intro}

The detection and mitigation of attacks on Cyber-Physical Systems (CPS) is a problem of increasing importance. The tight coupling between ``cyber'' components and ``physical'' processes often leads to systems where the increased sophistication comes at the expense of increased vulnerability and security weaknesses. 
An important scenario is posed by a malicious adversary that can arbitrarily corrupt the measurements of a subset of sensors in the system. These sensor-related attacks can be deployed in any of the following components of a real-life CPS:
\begin{enumerate}
	\item \emph{Software.} Malicious software running on the processor executing the sensor processing routine can access the sensor information before it is processed by the controller itself. 
	The Stuxnet malware is an infamous example of this category of attacks. It exploits vulnerabilities in the operating system running over SCADA devices \cite{stuxnet} and once it obtains enough operating system privileges, it can corrupt the sensor measurements collected via the attacked SCADA device.
	\item \emph{Network.} Modern control systems rely on a networked infrastructure to exchange sensor information. Therefore, 
	 an adversarial attacker can corrupt sensor measurements by manipulating the data packets exchanged between various components, as has been investigated, for instance, in smart grids~\cite{LiuPowerAttack}.
	\item \emph{Sensors Spoofing.}  By tampering with the sensor hardware and/or environment, an adversary can mislead the sensor about the value of the physical signal it is attempting to measure.
	As previously shown by some of the authors, it is possible to make drivers
lose control of their cars by directly spoofing the velocity sensors of anti-lock
braking systems in a non-invasive manner~\cite{YasserABS}.
\end{enumerate}

This paper addresses the problem of detecting and mitigating the effects of an adversarial corruption of sensory data in a linear dynamical system. While detection is concerned with determining which sensors are under attack, mitigation is concerned with the ability to estimate the state of the underlying physical system from corrupted measurements, so that it can be used by the controller. We call the latter problem \emph{secure state estimation}. We focus on linear dynamical systems and model the attack as a sparse vector added to the measurement vector. The entries corresponding to unattacked sensors are null while sensors under attack are corrupted by non-zero signals. We make no assumptions regarding the magnitude, statistical description, or temporal evolution of the attack vector.

%


While some prior work focused on the special cases of scalar system \cite{Bai_Gupta} and/or special structure on the attack signal (e.g. replay attacks in~\cite{YilinAllerton}), the work reported in this work focuses on the case when the underlying system is multi-dimensional, equipped with multiple sensors and without assumptions on the knolwedge of the time evolution of the attack signal. In such case, the \emph{secure state estimation} problem becomes a combinatorial problem~\cite{Hamza_TAC,Bullo_TAC,Joao_ACC}. We can broadly categorize the prior work in this area based on the technique used to tackle the combinatorial aspect of the problem into two wide categories (i) brute force search and (ii) convex relaxation.

The work reported in~\cite{Bullo_TAC, Joao_ACC} are representative of the first class; brute force search. Pasqualetti~et~al.~\cite{Bullo_TAC} provide a suite of sound and complete algorithms to generate fault-monitor filters, which can be used to detect the existence of an attack. However, if only an upper bound on the cardinality of the attacked sensors is available, the number of needed monitors is combinatorial in the size of the attacked sensors, which might hinder the scalability of the approach.  To avoid running a combinatorial set of parallel monitors, Chong~et~al.~\cite{Joao_ACC} shows how all the monitors can be combined into a single multi-observer with a combinatorial number of outputs. The algorithm reported in their work searches over all the outputs in order to discover which sensors are under attack.

For the convex relaxation approach, prior work reported in~\cite{Hamza_TAC} (for the case where sensors are ideal and not affected by noise) and~\cite{Pajic_ICCPS} (for the noisy case) shows how to formulate the secure state estimation problem as a non-convex $l_0$ minimization problem and then relax it into a convex a convex $l_r/l_1$ problem, which can be solved in polynomial time. The major drawback of this relaxation step is the loss of correctness guarantees. In particular, we show experimental results, at the end of this paper, for which the relaxed $l_r/l_1$  leads to incorrect results. To avoid the relaxation step while obtaining an algorithm that runs in polynomial time, an alternative formulation was proposed in~\cite{YasserCDC_ET,YasserETPGarXiv}.  However, correctness of the results obtained by the proposed algorithms are guaranteed only when restrictive conditions are satisfied by the system structure.

%
%
%

Another suite of algorithms are also proposed for the secure state estimation without any formal guarantees on their correctness. For example, a technique that relies on an on-line learning mechanism based on approximate envelopes of collected data has also been recently reported~\cite{Tiwari_HiCONS}. The envelopes are used to detect any abnormal behavior without assuming any knowledge of the dynamical system model. Another techniques are proposed in \cite{BoydKF} and \cite{Georgios_TSP2} in which 
robustification approaches for state estimation (using either Kalman Filters or Principal Component Analysis) against sparse sensor attacks are proposed, again with no guarantees on their correctness.


In this work, we resort to techniques from formal methods to develop a \emph{sound and complete} algorithm that can \emph{efficiently} handle the combinatorial complexity of the state estimation problem.  We show that the state estimation problem can be cast as a 
satisfiability problem for a formula including logic and pseudo-Boolean constraints on Boolean variables as well as convex constraints on real variables. 
The Boolean variables model the presence (or absence) of an attack, while the convex constraints capture properties of the system state. We then show how this satisfiability problem can be efficiently solved using the \emph{Satisfiability-Modulo-Theories} (SMT) paradigm~\cite{smtbook}, specifically adapted to convex constraint solving~\cite{CalCS}, to provide both the index of attacked sensors and the state estimate.
To improve the execution time of our decision procedure, we equip the convex constraint solver of our SMT-based algorithm with heuristics that can exploit the specific geometry of the state estimation problem. 
Finally, we  compare the  performance of our approach   against other algorithms via numerical experiments, and demonstrate its effectiveness on the problem of controlling an Unmanned Ground Vehicle (UGV).

Technically, we make the following contributions:
\begin{itemize}
\item We formalize the \emph{secure state estimation} problem as a satisfiability problem which includes both boolean constraints and convex constraints over real variables.
\item We provide \textsc{Imhotep\footnote{Imhotep: (pronounced as ``emmo-tepp'') was an ancient Egyptian polymath who is considered to be the earliest known architect, engineer and physician in the early history. He is famous of the design of the oldest pyramid in Egypt; Pyramid of Djoser (the Step Pyramid) at Saqqara, Egypt, 2630 -- 2611 BC.}-SMT}; a novel SMT-solver that is shown, formally, to provide a sound and complete solution to the \emph{secure state estimation} problem.
\item We propose heuristics to improve the execution time of the \textsc{Imhotep-SMT} solver along with  the real-time guarantees given as the upper bounds on the number of iterations required by the proposed algorithm.
\end{itemize}

The rest of this paper is organized as follows. Section \ref{sec:problem} introduces the formal setup for the problem under consideration. The main contributions of this paper -- the introduction of the SMT-based detector and the characterization of its soundness and completeness -- are presented in Section \ref{sec:smt} and Section \ref{sec:comp}. Numerical comparisons and results are then shown in Section \ref{sec:results}. Finally, Section \ref{sec:conclusion} concludes the paper and discusses new research directions. 

\section{The Secure State Estimation Problem}
\label{sec:problem}

We provide a mathematical formulation of the state estimation problem considered in this paper and discuss the conditions for the existence and uniqueness of the solution.
\subsection{Notation}
The symbols $\N, \R$ and $\B$ denote the sets of natural, real, and Boolean numbers, respectively. The symbols $\land$ and $\lnot$ denote the logical AND and logical NOT operators, respectively.
The support of a vector $x\in \R^n$,  denoted by $\supp(x)$, is the set of indices of the non-zero elements of $x$. Similarly, the complement of the support of a vector $x$ is denoted by $\overline{\supp(x)} = \{1,\hdots,n\}\setminus\supp(x)$.
If $S$ is a set, $|S|$ is the cardinality of $S$. 
We call a vector $x \in \R^n$ $s$-sparse, if  $x$ has at most $s$ nonzero
elements, i.e., if  $\vert \supp(x)\vert \le s$. 
For a vector $x \in \R^n$, we denote by $\norm{x}$ the $2$-norm of $x$  and by $\norm{M}$ the induced $2$-norm of a matrix $M \in \R^{m \times n}$. We also denote by $M_i \in \R^{1\times n}$ the $i${th} row of $M$. For the set \mbox{$\Gamma \subseteq \{1, \hdots, m\}$}, we denote by
$M_{\Gamma} \in \R^{\vert \Gamma\vert \times n}$ the matrix obtained from $M$ by  removing all the rows except those
indexed by $\Gamma$. Then, $M_{\overline{\Gamma}} \in
\R^{(m-|\Gamma|) \times n}$ is the matrix obtained from $M$ by removing the rows indexed by
the set $\Gamma$, $\overline{\Gamma}$ representing the complement of $\Gamma$. For example, if $m = 4$, and $\Gamma = \{1,2\}$, we have
$$M_\Gamma =\matrix{M_1 \\ M_2}, \quad M_{\overline{\Gamma}} =
\matrix{M_3\\M_4}.$$


\subsection{System and Attack Model}

We consider a system under sensor attack of the form:
\begin{align}
\Sigma_a  \quad
\begin{cases}
	x^{(t+1)} &= A x^{(t)} + B u^{(t)}, \\
	y^{(t)} &= C x^{(t)} + a^{(t)} + \psi^{(t)}
\end{cases}
\label{eq:system_attack}
\end{align} 
where $x^{(t)} \in \R^n$ is the system state at time  $t \in \N$, $u^{(t)} \in \R^m$ is the system input, and $y^{(t)} \in \R^p$ is the observed output. The matrices $A, B$, and $C$ represent the system dynamics and have appropriate dimensions. 
The attack vector $a^{(t)} \in \R^p$ is an $s$-sparse vector modeling how an attacker changed the sensor measurements at time $t$.  If sensor $i \in \{1, \hdots ,p\}$ is attacked then
the $i${th} element in $a^{(t)}$ is non-zero; otherwise the $i${th}
sensor is not attacked. Hence, $s$ describes the number of attacked sensors. 
Note that we make no assumptions on the vector $a^{(t)}$ apart from being $s$-sparse. 
In particular, we do not assume bounds, statistical properties, nor restrictions on 
the time evolution of the elements in $a^{(t)}$. The value of $s$ is also not
assumed to be known, although we assume the knowledge of an upper bound $\overline{s}$ on the maximum number of sensors that can be attacked.
Finally, the vector $\psi^{(t)} \in \R^p$ represents the measurement noise, which is assumed to be bounded.

\subsection{Problem Formulation} \label{sec:formulation}

To formulate the state estimation problem, we assume the state is recountsucted from a set of $\tau\in \N$ measurements,  where $\tau \le n$ is selected to guarantee that  the system observability matrix, as defined below, has full rank. Therefore, we can arrange the outputs from the $i${th} sensor at different time instants as follows:
$$
\tilde{Y}_i^{(t)} = \O_{i} x^{(t-\tau+1)} + E_i^{(t)} + F_i U^{(t)} + \Psi_i^{(t)}, 
$$
where:
\begin{align*}
	\tilde{Y}_i^{(t)} &= \matrix{
		y_i^{(t-\tau+1)} \\ y_i^{(t-\tau)} \\ \vdots \\ y_i^{(t)}
	}, 
	E_i^{(t)}  =\matrix{ 
		a_i^{(t-\tau+1)} \\ a_i^{(t-\tau)} \\ \vdots \\ a_i^{(t)}
	}, 
	\Psi_i^{(t)}  =\matrix{ 
		\psi_i^{(t-\tau+1)} \\ \psi_i^{(t-\tau)} \\ \vdots \\ \psi_i^{(t)}
	}, 
	U^{(t)} = \matrix{ u^{(t-\tau+1)} \\ u^{(t-\tau+2)} \\ \vdots \\ u^{(t)} },
	\O_i = \matrix{ C_i \\ C_iA \\ \vdots \\ C_iA^{\tau-1} },\\
	F_i &= \matrix{ 0 & 0 & \hdots & 0 & 0\\ C_i B & 0 & \hdots & 0 & 0 \\ \vdots &
	& \ddots &  & 
	\vdots
	\\
	C_i A^{\tau-2}B & C_i A^{\tau-3}B & \hdots & C_i B & 0 }
	.
\end{align*}
Since all the inputs in $U^{(t)}$ are known, we can further simplify
the output equation as:
\begin{equation} \label{eq:y_equality}
Y_i^{(t)}=\O_{i} x^{(t-\tau+1)} + E_i^{(t)} + \Psi_i^{(t)},
\end{equation}
where $Y_i^{(t)}=\tilde{Y}_i^{(t)}-F_iU^{(t)}$. We also define:
\begin{align}
 Y^{(t)} = \matrix{Y_1^{(t)} \\ \vdots \\ Y_p^{(t)}}, \quad E^{(t)} = \matrix{E_1^{(t)} \\ \vdots \\ E_p^{(t)}}, \quad \O = \matrix{\O_{1} \\ \vdots \\ \O_{p}}
 \label{eq:O}
 \end{align}
to denote, respectively, the vector of outputs, attacks and observability matrices related to all sensors over the same time window of length $\tau$.  Here, with some abuse of notation, $Y_i, E_i$ and $\O_{i}$
are used to denote the $i${th} block of $Y, E$ and $\O$. Then, by the same abuse of notation, we also denote by $Y_{\Gamma}, E_{\Gamma}$, and $\O_{\Gamma}$ the blocks indexed by the elements in the set $\Gamma$.

\subsection{Problem Statement}
For each individual sensor, we define a binary indicator variable $b_i \in \B$ such that $b_i = 0$ when the $i${th} sensor is attack-free and $b_i = 1$ otherwise. Based on the formulation in Sec.~\ref{sec:formulation}, our goal is to find $x^{(t-\tau+1)}$ in~\eqref{eq:y_equality}, knowing that:
\begin{enumerate}
\item if a sensor is attack-free (i.e., $b_i = 0$), then \eqref{eq:y_equality} reduces to  $Y_i^{(t)} - \O_{i} x^{(t-\tau+1)} = \Psi_i^{(t)}$; 
\item $\Psi_i$ is the upper bound on the norm of the noise at sensor $i$, i.e.,
$$\norm{\Psi_i^{(t)}}\le \Big\Vert \Psi_i\Big\Vert_2, \quad \forall t \in \N$$
\item the maximum number of attacked sensors is $\overline{s}$. 
 \end{enumerate}
 Therefore, using the binary variables $b_i$, we can pose the problem of secure state estimation as follows.

\begin{problem}{\textbf{(Secure State Estimation)}}
For the linear control system under attack $\Sigma_a$ (defined by~\eqref{eq:system_attack}), construct an estimate $\eta = (x, b) \in \R^n \times \B^p$  such that $\eta \models \phi $, i.e., $\eta$ satisfies the formula $\phi$, where $\phi$ is defined as:
\begin{align*}  
\phi ::= &\bigwedge_{i = 1}^p  \Bigg(\lnot b_i  \Rightarrow \norm{Y_i - \O_{i} x} \le \norm{\Psi_i} \Bigg) \quad \bigwedge \left(\sum_{i = 1}^p b_i \le \overline{s}\right). \notag
\end{align*}
\label{prob:sse}
\end{problem}
\noindent  In Problem~\eqref{prob:sse}, $Y_i$, $\norm{\Psi_i}$ and $\O_{i}$ are the vectors of outputs, measurement noise bound and the observability matrix related to sensor $i$, as defined in Sec.~\ref{sec:formulation}. The first conjunction of constraints requires $(Y_i - \O_{i} x)$ to be bounded only by the noise bound if sensor $i$ is attack-free. We resort to the 2-norm of $(Y_i - \O_{i} x)$ since the only information we have available about the noise is a bound on its 2-norm. The second inequality enforces the cardinality constraint on the number of attacked sensors. 

We drop the time $t$ argument in Problem~\eqref{prob:sse} since the satisfiability problem is to be solved at every time instance.  Note that although we reconstruct a delayed version of the state $x^{(t-\tau+1)}$, we can always reconstruct the current state $x^{(t)}$ from $x^{(t-\tau+1)}$ by recursively rolling the dynamics forward in time. 

The secure state estimation problem~\ref{prob:sse} does not ask for the minimal number of attacked sensors for which the estimated state matches the measured output. That is, if $b^*$ is the vector of indicator variables characterizing the actual attack, any assignment $\eta = (x, b) \models \phi$ with $\supp(b^*) \subseteq \supp(b)$ is a valid solution for Problem~\ref{prob:sse}.
Therefore, it is useful to modify Problem~\ref{prob:sse} to ask for the minimal number of attacked sensors that explains the collected measurements as follows.

\begin{problem}{\textbf{(Minimal Attack Support)}}
For the linear control system under attack $\Sigma_a$ construct the estimate \mbox{$\eta = (x, b) \in \R^n \times \B^p$}  obtained as the solution of the optimization problem:
\begin{align*}
\min_{(x,b)\in \R^n \times \B^p} & \; \sum_{i = 1}^{p}  {b}_i \qquad
\text{s.t.}  \; \bigwedge_{i =1}^p \Bigg(\lnot {b}_i  \Rightarrow \norm{Y_i - \O_{i} x } \le \norm{\Psi_i} \bigg).
\end{align*}
\label{prob:mas}
\end{problem}

It is straightforward to show that the solution to Problem \ref{prob:mas} can be obtained by performing a binary search over $s$ and invoking a solver for Problem~\ref{prob:sse} at each step, starting with $s = \overline{s}$ and then decreasing $\overline{s}$ until Problem~\ref{prob:sse} becomes unfeasible or $\overline{s}=0$. Since any solution of~\eqref{eq:y_equality}  must necessarily satisfy the constraints of Problem~\ref{prob:sse}, such a procedure will terminate by returning the solution with the minimal attack support. We denote this solution as \emph{minimal support solution}. In the reminder of the paper, we will focus on the analysis of the feasibility problem~\ref{prob:sse}, since a solution to the optimization problem~\ref{prob:mas} can be obtained by solving a sequence of instances of Problem~\ref{prob:sse}. 

In Sec.~\ref{sec:guarantees}, we discuss the conditions for the uniqueness of the minimal support solution 
of  Problem~\ref{prob:mas}.
However, we first recall that the satisfiability problem over real numbers, and specifically~over $\R^n$, is inherently intractable, i.e., decision algorithms for formulas with non-linear polynomials already suffer from high complexity~\cite{BrownRealSAT,CollinsRealSAT}. 
Moreover, linear programming and convex programming solvers usually perform floating point (hence inexact) calculations, which may be inadequate for some applications.
Therefore, to provide formal guarantees about correctness of Problem~\ref{prob:sse}, we resort to the notion of $\delta$-completeness which was previously used in~\cite{deltaComplete}.

\begin{definition}[\textbf{Soundness and Completeness of Decision Algorithms for Problem \ref{prob:sse}}]
\label{def:deltacom}
Let a minimal solution \mbox{$\eta^* = (x^*,b^*)$} (the true state and indicator variables) exist for Problem~\ref{prob:sse}.  Then, 
a solution $\eta = (x,b) \models \phi$ is said to $\delta$-satisfy $\phi$ (or $\delta$-SAT for short) if $\supp(b^*) \subseteq \supp(b)$ and $\norm{x^* - x}^2 \le \delta$ for some $\delta \in \R$. Moreover, an algorithm that solves Problem~\ref{prob:sse} is said to be $\delta$-\emph{complete} if it returns a $\delta$-SAT solution.
\end{definition}

Definition \ref{def:deltacom} asks for an algorithm which terminates and returns a solution $\eta = (x, b)$ that is correct (up to the tolerance $\delta$). Hence, a $\delta$-\emph{complete} decision algorithm in the sense of Definition \ref{def:deltacom} is also ($\delta$-)sound since, if it returns a solution $\eta$, 
$\eta$ is actually a $\delta$-SAT solution.


\subsection{Uniqueness of Minimal Support Solutions} \label{sec:guarantees}

To characterize the existence and uniqueness of solutions to Problem \ref{prob:mas}, 
we recall  the notion of $s$-sparse observability~\cite{YasserETPGarXiv}.

\begin{definition}{\textbf{($s$-Sparse Observable System)}}
The linear control system $\Sigma_a$, defined by~\eqref{eq:system_attack},
is said to be $s$-sparse observable
if for every set $\Gamma\subseteq\{1,\hdots,p\}$ with $|\Gamma| = s$, the system $\Sigma_{\overline{\Gamma}} $ is observable, where $\Sigma_{\overline{\Gamma}}$ is defined as:\begin{align}
\Sigma_{\overline{\Gamma}}   \quad
\begin{cases}
	x^{(t+1)} &= A x^{(t)} + B u^{(t)}, \qquad t \in \N\\
	y^{(t)} &= C_{\overline{\Gamma}} x^{(t)}
\end{cases}.
\end{align} 
\end{definition}
\noindent In other words, a system is $s$-sparse observable if it is  observable from any choice of $p  - s$ sensors. For $2\overline{s}$-sparse observable systems, the following result holds. 

\begin{theorem}{\textbf{(Existence and Uniqueness of the Solution)}}[Theorem III.2 in \cite{YasserETPGarXiv}] 
In the noiseless case ($\Psi_i = 0$ for all $i\in \{1,\ldots, p\}$),  Problem \ref{prob:mas} admits a unique solution $\eta^* = (x^*,b^*)$ if and only
if the dynamical system $\Sigma_a$ defined by \eqref{eq:system_attack} is $2\overline{s}$-sparse observable.
\label{th:sparse}
\end{theorem}

The following result was established as part of the proof of Theorem \ref{th:sparse} in~\cite{YasserETPGarXiv} and will be used in the Section\ref{sec:smt}.
\begin{proposition}
Let the dynamical system $\Sigma_a$ defined by \eqref{eq:system_attack} be $2\overline{s}$-sparse observable. The observability matrix $\O_{\mathcal{I}}$ has a trivial kernel for any set \mbox{$\mathcal{I} \subseteq\{1,\ldots,p\}$} with \mbox{$\vert {\mathcal{I}}\vert \ge p - 2\overline{s}$}.
\label{cor:injectivity}
\end{proposition}

\begin{remark}
As stated in Theorem~\ref{th:sparse}, the state of $\Sigma_a$ can be uniquely determined when the system
is $2\overline{s}$-sparse observable. This condition seems expensive to check because of its combinatorial nature: we have to check observability of all possible systems $\Sigma_{\overline{\Gamma}}$. Yet,  the $2\overline{s}$-sparse observability condition clearly illustrates a fundamental limitation for secure state estimation: \emph{it is impossible to correctly reconstruct the state whenever a number of sensors larger than or equal to $\lceil p/2 \rceil$ is attacked, since multiple states can be mapped to the same measurements}. 

Indeed, suppose that we have
an even number of sensors $p$ and $\overline{s} = p/2$ sensors are attacked. Then, Theorem~\ref{th:sparse}  requires  the system to still be observable after removing
$2\overline{s} = p$ rows from the map $C$. However, this is impossible since $C_{\overline{\Gamma}}$ becomes the 
transformation mapping every state to zero.  
This fundamental
limitation is consistent with previous results reported in the literature~\cite{Hamza_TAC,HamzaAllerton,strongObservability}.
\end{remark}

Problem \ref{prob:mas} can be solved by transforming it into a Mixed Integer-Quadratic Program (MIQP) as follows:
\begin{align} \label{eq:miqp}
{\min_{(x,b) \in \R^n \times \B^p}} & \;\sum_{i=1}^{p} b_{i} \qquad
\text{s.t.} 
&\; \|{Y_{i} - \O_i x}\|_{2}\leq Mb_{i}+\norm{\Psi_{i}}& 1\leq i\leq p, 
\end{align}
where $M \in \R$ is a constant that should be ``big'' enough to make each constraint not active when $b_i=1$. The relaxation in~\eqref{eq:miqp} is typically used to express constraints including logical implications~\cite{l2008operations}; however, in this case, the choice of $M$ affects the completeness of the approach, which will depend on $M$. For example, since $\|{Y_{i} - \O_i x}\|_{2}$ is ultimately bounded by the power of the attack $\norm{E_i}$, a value of $M < \norm{E_i} = \norm{Y_{i} - \O_i x}$, in the absence of noise, can produce an incorrect result. While a physical sensor has a bounded dynamic range in practice, such a bound is not known \emph{a priori} in our formulation, which makes no assumptions on $\norm{E_i}$. Therefore, completeness of the MIQP formulation~\eqref{eq:miqp} cannot be guaranteed in general. 

In the sequel, we detail an algorithm which exploits the geometry of the state estimation problem and the convexity of the quadratic constraints to generate a provably correct solution using the SMT paradigm. We compare the SMT-based solution with the MIQP formulation in~\eqref{eq:miqp} using a commercial MIQP solver.

\section{SMT-Based Detector}
\label{sec:smt}

To decide whether a combination of Boolean and convex constraints is satisfiable, we construct the detection algorithm \textsc{Imhotep-SMT} using the \emph{lazy} SMT paradigm~\cite{smtbook}. As in the CalCS solver~\cite{CalCS}, our 
decision procedure combines a SAT solver (\textsc{SAT-Solve}) and a theory solver ($\mathcal{T}$-\textsc{Solve}) for convex constraints on real numbers. The SAT solver efficiently reasons about combinations of Boolean and pseudo-Boolean constraints, using the 
David-Putnam-Logemann-Loveland (DPLL) algorithm~\cite{lazySMT}, to suggest possible assignments for the convex constraints. The theory solver checks the consistency of the given assignments, and provides the reason for the conflict, a \emph{certificate}, or a counterexample, whenever inconsistencies are found. Each certificate results in learning new constraints which will be used by the SAT solver to prune the search space. The complex detection and mitigation decision task is thus broken into two simpler tasks, respectively, over the Boolean and convex domains. We denote the approach as lazy, because it checks and learns about consistency of convex constraints only when necessary, as detailed below.

\subsection{Overall Architecture}

As illustrated in Algorithm~\ref{alg:smt},  we start by mapping each convex constraint to an auxiliary Boolean variable $c_i$ to obtain the following (pseudo-)Boolean satisfiability problem:
\begin{align*}
\phi_B ::= \left( \bigwedge_{i \in \{ 1,\hdots,p \}}  \lnot b_i  \Rightarrow c_i \right) 
\land \left(\sum_{i \in \{ 1,\hdots,p \}} b_i \le \overline{s} \right)
\end{align*}
where $c_i = 1$ if $\norm{Y_i - \O_{i} {x} } \le \norm{\Psi_i}$ is satisfied, and zero otherwise. By only relying on the Boolean structure of 
the problem, \textsc{SAT-Solve} returns an assignment for the variables $b_i$ and $c_i$ (for $i = 1, \hdots, p$), thus hypothesizing which sensors are attack-free, hence which convex constraints should be jointly satisfied.  

This Boolean assignment is then used by $\mathcal{T}$-\textsc{Solve} to determine whether there exists a state $x \in \R^n$ which satisfies all the convex constraints related to the unattacked sensors, i.e.~$\norm{Y_i - \O_{i} {x} } \le \norm{\Psi_i}$ for $i \in \overline{\supp}(b)$. If $x$ is found, \textsc{Imhotep-SMT} terminates with $\verb+SAT+$ and provides the solution $(x,b)$. Otherwise, the $\verb+UNSAT+$ certificate $\phi_{\text{cert}}$ is generated in terms of new Boolean constraints, explaining which sensor measurements are conflicting and may be under attack. The most naive certificate can take the form of:
\begin{align*}
 \phi_{\text{UNSAT-cert}} = \sum_{i \in \overline{\supp}(b)} b_i \ge 1,
\end{align*}
which encodes the fact that at least one of the sensors in the set $\overline{\supp}(b)$ (i.e. for which $b_i=0$) is actually under attack.
This augmented Boolean problem is then fed back to \textsc{SAT-Solve} to produce a new assignment. The sequence of new SAT queries is then repeated until $\mathcal{T}$-\textsc{Solve} terminates with $\verb+SAT+$.

By the $2\overline{s}$-sparse observability condition (Theorem \ref{th:sparse}), there always exists a unique solution to Problem~\ref{prob:mas}, hence Algorithm~\ref{alg:smt} will always terminate. However, to help the SAT solver quickly converge towards the correct assignment, a central problem in lazy SMT solving is to generate succinct explanations whenever conjunctions of convex constraints are unfeasible, possibly highlighting the minimum set of conflicting assignments.  
The rest of this section will then focus on the implementation of the two main tasks of $\mathcal{T}$-\textsc{Solve}, namely, (i) checking the satisfiability of a given assignment ($\mathcal{T}$-\textsc{Solve.Check}), and (ii) generating succinct UNSAT certificates ($\mathcal{T}$-\textsc{Solve.Certificate}). For clarity's sake, we focus on the noiseless case ($\Psi = 0$) in this section; we will extend our results to the noisy case in Section~\ref{sec:comp}.

\begin{algorithm}[t]
\caption{\textsc{Imhotep-SMT}}
\begin{algorithmic}[1]
\STATE status := \verb+UNSAT+;
\STATE $\phi_B := \left( \bigwedge_{i \in \{ 1,\hdots,p \}}  \lnot b_i  \Rightarrow c_i \right) \land \left( \sum_{i \in \{ 1,\hdots,p \}} b_i \le \overline{s} \right)
$;
\WHILE{status == \texttt{UNSAT}}
	\STATE $(b,c) :=$ \textsc{SAT-Solve}$( \phi_B )$;
	\STATE (status, $x$) := $\mathcal{T}$-\textsc{Solve.Check}$(\overline{\supp}(b))$;
	\IF {status == \texttt{UNSAT}}
		\STATE $\phi_{\text{cert}}$ := $\mathcal{T}$-\textsc{Solve.Certificate}$(b, x)$;
		\STATE $\phi_B := \phi_B \land \phi_{\text{cert}}$;
	\ENDIF
\ENDWHILE
\STATE \textbf{return} $\eta = (x,b)$;
\end{algorithmic}
\label{alg:smt}
\end{algorithm}

\subsection{Satisfiability Checking}
Given an assignment of the Boolean variables $b$, with $\vert\supp(b)\vert \le \overline{s}$, the following condition holds:
\begin{equation}\label{eq:lmq}
\min_{x \in \R^n} \norm{Y_{\overline{\supp}(b)} - \O_{\overline{\supp}(b)}x}^2 \le 0
\end{equation}
if and only if $(x,b)$ is the solution of Problem~\ref{prob:mas}. This is a direct consequence of the $2\overline{s}$-sparse observability property discussed in Section~\ref{sec:problem}. 
The preceding \emph{unconstrained least-squares optimization} problem can be solved very efficiently, thus leading to Algorithm~\ref{alg:check}. In practical implementations, \eqref{eq:lmq} should actually be replaced with:
$$ \min_{x \in \R^n} \norm{Y_{\overline{\supp}(b)} - \O_{\overline{\supp}(b)}x}^2 \le \epsilon,$$
where $\epsilon > 0$ is the solver tolerance, accounting for numerical errors. As for noise, for the sake of clarity, we focus here on the case when $\epsilon$ is zero and defer the discussion for non-zero tolerance to the next section.

Since Algorithm~\ref{alg:check} is the basic block of our SMT-based detector, it is important to characterize its soundness and completeness, as is done in the following result.

\begin{lemma}
Let the linear dynamical system $\Sigma_a$ defined in~\eqref{eq:system_attack} be $2\overline{s}$-sparse observable. Let $\norm{\Psi_i} = 0$ for all $i \in \{1,\hdots,p\}$ and let also $\epsilon = 0$ be the numerical solver tolerance for Algorithm \ref{alg:check}. Then for any index set $\mathcal{I}$ with cardinality \mbox{$\vert \mathcal{I} \vert \ge p - \overline{s}$}, Algorithm \ref{alg:check} returns \texttt{SAT} if and only if the following holds:
\begin{enumerate}
\item  $\mathcal{I} \subseteq \overline{\supp}(b^*)$,
\item  $\norm{x^* - x}^2 = 0$,
\end{enumerate}
where $(x^*, b^*)$ is the solution to Problem \ref{prob:mas} and $x$ is computed on line 1 of Algorithm \ref{alg:check}.
\label{prop:check}
\end{lemma}
\begin{proof}
Since the ``if'' condition is trivial to show, we focus on the ``only if'' condition as follows. Define $\mathcal{I}'$ as the set of indices of the sensors that are attack free. Define also $\mathcal{I}''$ as the set $\mathcal{I}'' = \mathcal{I} \setminus \mathcal{I}'$.  We can write the result from lines 1 and 2 of Algorithm~\ref{alg:check} as:
\begin{align*}
\min_{{x} \in \R^n} &\norm{Y_{\mathcal{I}} - \O_{\mathcal{I}} {x} }^2 = 0 \\
 	&\Rightarrow \min_{x \in \R^n} \sum_{i \in {\mathcal{I}}} \norm{Y_i - \O_i x }^2 = 0\\
	&\Rightarrow  \min_{x \in \R^n} \sum_{i \in {\mathcal{I'}}} \norm{Y_i - \O_i x }^2 + \sum_{i \in {\mathcal{I''}}} \norm{Y_i - \O_i x }^2 = 0\\
	&\Rightarrow \min_{x \in \R^n} \norm{\O_{\mathcal{I'}} (x^* - {x}) }^2 + \sum_{i \in {\mathcal{I''}}} \norm{\O_i (x^* - x) + E^*_i }^2 = 0
\end{align*}
Hence, in order for Algorithm~\ref{alg:check} to return \texttt{SAT}, both terms $\norm{\O_{\mathcal{I'}} (x^* - {x}) }^2$ and $\sum_{i \in {\mathcal{I''}}} \norm{\O_i (x^* - x) + E^*_i }^2$ must vanish at the optimal point.

Since at most $\overline{s}$ sensors are under attack, we conclude that  $\vert \mathcal{I}'' \vert$ is at most $\overline{s}$ and $\vert \mathcal{I}'  \vert \ge p - 2\overline{s}$. Hence, it follows from  Proposition~\ref{cor:injectivity} that the observability matrix $\O_{\mathcal{I'}}$ has a trivial kernel. Therefore, we conclude that $\norm{\O_{\mathcal{I'}} (x^* - {x}) }^2$ evaluates to zero if and only if $x = x^*$. This, in turn, implies that the solution of the optimization problem in line 1 of Algorithm~\ref{alg:check} is $x^*$ and hence $\norm{x^* - x}^2 = 0$.

To conclude, we need to show that $\mathcal{I} \subseteq \overline{\supp}(b^*)$. However, this follows from the requirement that $\sum_{i \in {\mathcal{I''}}} \norm{\O_i (x^* - x) + E^*_i }^2$ vanishes at the optimal point, i.e.,~for $x  = x^*$. Hence:
\begin{align*}
\sum_{i \in {\mathcal{I''}}} \norm{\O_i (x^* - x) + E^*_i }^2 &= 0 
 \Rightarrow \sum_{i \in {\mathcal{I''}}} \norm{E^*_i }^2 = 0 
\end{align*}
which, in turn, implies that all the sensors indexed by $\mathcal{I''}$ are attack free. Combining this result with the definition of the set $\mathcal{I'}$ we conclude that all the sensors indexed by $\mathcal{I}$ are actually attack free, and the inclusion $\mathcal{I} \subseteq \overline{\supp}(b^*)$ holds.
\end{proof}

When noise and/or non-zero numerical tolerance is present, we modify Algorithm~\ref{alg:check} by checking instead whether the optimal $x$ drives the objective function below the noise level and the numerical tolerance. Clearly, satisfying such a constraint on the $2$-norms is not sufficient to retrieve the actual state in the sense of Definition~\ref{def:deltacom}: attacks having a relatively small power may not be detected. Therefore, in Section~\ref{sec:comp}, we will determine which conditions to require on the noise level and the numerical tolerance to achieve $\delta$-completeness as in Definition~\ref{def:deltacom}.

\begin{algorithm}[t]
\caption{$\mathcal{T}$\textsc{-Solve.Check}$(\mathcal{I})$}
\begin{algorithmic}[1]
\STATE \textbf{Solve:} $x:= \argmin_{{x} \in \R^n} \norm{Y_{\mathcal{I}} - \O_{\mathcal{I}} {x} }^2$
\IF{ $\norm{Y_{\mathcal{I}} - \O_{\mathcal{I}} x}^2 = 0$} \label{line:check}
	\STATE status = \texttt{SAT};
\ELSE
	\STATE status = \texttt{UNSAT};
\ENDIF
\STATE \textbf{return} (status, $x$);
\end{algorithmic}
\label{alg:check}
\end{algorithm}

\subsection{Generating UNSAT certificates}

Whenever $\mathcal{T}$\textsc{-Solve.Check} provides \texttt{UNSAT}, a certificate could be easily generated as follows:
\begin{align}
 \phi_{\text{triv-cert}} = \sum_{i \in \overline{\supp}(b)} b_i \ge 1,
\label{eq:trivcert} 
\end{align}
indicating that at least one of the sensors, which was initially assumed as attack-free (i.e. for which $b_i=0$), is actually under attack; one of the $b_i$ variables should then be set to one in the next assignment of the SAT solver.  However, such \emph{trivial certificate} $\phi_{\text{triv-cert}}$ does not provide much information, since it only excludes the current assignment from the search space, and can lead to exponential execution time, as reflected by the following proposition.

\begin{proposition}
Let the linear dynamical system $\Sigma_a$ defined in~\eqref{eq:system_attack} be $2\overline{s}$-sparse observable. Let $\norm{\Psi_i} = 0$ for all $i \in \{1,\hdots,p\}$ and let also 
$\epsilon = 0$ be the numerical solver tolerance for Algorithm \ref{alg:check}. Then, Algorithm~\ref{alg:smt} which uses the trivial UNSAT certificate $\phi_{\text{triv-cert}}$ in~\eqref{eq:trivcert} is $\delta$-complete (in the sense of Definition \ref{def:deltacom}) with $\delta = 0$.
Moreover, the upper bound on the number of iterations of Algorithm~\ref{alg:smt} is $ \sum_{s = 0}^{\overline{s}} \binom{p}{s}.$
\label{prop:trivial-cert}
\end{proposition}

\begin{proof}
$\delta$-Completeness of Algorithm~\ref{alg:smt} follows directly from Lemma \ref{prop:check}. To derive the bound on the number of iterations, we first recall that the $2\overline{s}$-sparse observability condition ensures uniqueness of a minimal solution (Theorem~\ref{th:sparse}). The worst case scenario would happen when the solver exhaustively explores all possible combinations of attacked sensors with cardinality less than or equal to $\overline{s}$ in order to find the correct assignment. This is equal to $ \sum_{s = 0}^{\overline{s}} \binom{p}{s}$ iterations.
\end{proof}

\subsection{Enhancing the Execution Time} \label{sec:heur}

The generated UNSAT certificates heavily affects the overall execution time of Algorithm~\ref{alg:smt}: the  smaller the certificate, the more information is learnt and the faster is the convergence of the SAT solver to the correct assignment. For example,  a certificate with $b_i = 1$ would identify exactly one attacked sensor at each step, a substantial improvement with respect to the exponential worst-case complexity of the plain SAT problem, which is NP-complete. Hence, and 
inspired by the theoretical underpinnings of \textsc{CalCS}~\cite{CalCS}, we focus on designing heuristics that can lead to more \emph{compact certificates} to enhance the execution time of \textsc{Imhotep-SMT}. To do so, we exploit the specific structure of the secure state estimation problem and generate customized, yet stronger, UNSAT certificates. In this subsection, we focus on generating two types of certificates called (i) conflicting certificates and (ii) agreeable certificates.

First, we observe that the measurements of each sensor \mbox{$Y_i = \O_i x$} define a hyperplane $\mathbb{H}_i \subseteq \R^n$ as:
$$ \mathbb{H}_i  = \{ x \in \R^n \; \vert  \; Y_i - \O_i x = 0\}.$$
The dimension of $\mathbb{H}_i$ is given by the dimension of the null space of the matrix $\O_i$, i.e.,~$\text{dim}(\mathbb{H}_i) = \text{dim}(\ker{\O_i})$.
Then, satisfiability checking in Algorithm~\ref{alg:check} can be reformulated as follows. Let $r_i$ be the \emph{residual} of the state $x$ with respect to the hyperplane $\mathbb{H}_i$, defined as $r_i(x) = \norm{Y_i - \O_i x}^2$. The optimization problem in Algorithm~\ref{alg:check} is equivalent to searching for a point $x$ that minimizes the individual residuals with respect to all the hyperplanes $\mathbb{H}_i$ for $i \in {\mathcal{I}}$, i.e.: $$\min_{x \in \R^n} \norm{Y_{\mathcal{I}} - \O_{\mathcal{I}} x }^2 =  \min_{x \in \R^n} \sum_{i \in {\mathcal{I}}} \norm{Y_i - \O_i x }^2 = \min_{x \in \R^n} \sum_{i \in {\mathcal{I}}} r_i(x).$$

Based on the formulation above, it is straightforward to state and show the following result.
\begin{proposition}
Let the linear dynamical system $\Sigma_a$ defined in~\eqref{eq:system_attack} be $2\overline{s}$-sparse observable.  Let $\norm{\Psi_i} = 0$ for all $i \in \{1,\hdots,p\}$ and let also 
$\epsilon = 0$ be the numerical solver tolerance for Algorithm \ref{alg:check}. Then for any set of indices $\mathcal{I} \subseteq \{1, \hdots, p\}$,
the following statements are equivalent:
\begin{itemize}
\item $\mathcal{T}$\textsc{-Solve.Check}$(\mathcal{I})$ returns \texttt{UNSAT},
\item $\min_{x \in \R^n} \sum_{i \in {\mathcal{I}}} r_i(x) > 0$,
\item $\bigcap_{i \in \mathcal{I}} \mathbb{H}_i = \emptyset$.
\end{itemize}
\label{prop:geom}
\end{proposition}

\begin{figure*}[t]
\centering
{
\begin{tabular}{c|c}
	\subfigure[ Four hyperplanes corresponding to measurements from 4 different sensors. The red hyperplane corresponds to the sensor under attack. All other hyperplanes intersect at the unique  solution. The optimal point is marked as a black box.
	]{\label{fig:heur1} \resizebox{0.29\textwidth}{!}{
	\begin{tikzpicture}
		\begin{axis}[xlabel=$x_1$,ylabel={$x_2$}, xmin=0, xmax=5,ymax = 10]
			\addplot [mark=*,color=blue]{-x + 8};
			\addplot [mark=*,color=brown]{x + 4};
			\addplot [mark=square*,color=red]{6 + 2};			
			\addplot [mark=*,color=gray]{2*x + 2};
			\addplot  [mark=square*, style={solid, fill=black}] table {data.dat};
		\end{axis}
	\end{tikzpicture}}
	}&
	\subfigure[An example of a run of Algorithm \ref{alg:certificate1}. In the first iteration (left), the set $\mathcal{I}\_min\_r$ contains the $p - 2\overline{s} = 4 - 2\times1 = 2$ indexes of the sensors that correspond to the minimal residuals. This set is a non-conflicting set and hence the corresponding hyperplanes have a unique intersection point. In the second iteration (right), the index of the sensor corresponding to the maximum residual is added to the set $\mathcal{I}_{temp}$ resulting into a conflict. Algorithm \ref{alg:certificate1} terminates and returns the conflicting set  $\mathcal{I}_{temp}$ (In both cases, the optimal point is marked as a black box).
	]{\label{fig:timelapse1} \resizebox{0.58\textwidth}{!}{
	\begin{tikzpicture}	
		\begin{axis}[xlabel=$x_1$,ylabel={$x_2$}, xmin=0, xmax=5,ymax = 10]
			\addplot [mark=*,color=blue]{-x + 8};
			\addplot [mark=*,color=gray]{2*x + 2};
			\addplot  [mark=square*, style={solid, fill=black}] table {data2.dat};
		\end{axis}
	\end{tikzpicture}
	\begin{tikzpicture}	
		\begin{axis}[xlabel=$x_1$,ylabel={$x_2$}, xmin=0, xmax=5,ymax = 10]
			\addplot [mark=*,color=blue]{-x + 8};
			\addplot [mark=square*,color=red]{6 + 2};			
			\addplot [mark=*,color=gray]{2*x + 2};
			\addplot  [mark=square*, style={solid, fill=black}] table {data3.dat};			
		\end{axis}
	\end{tikzpicture}
	}
	}
\end{tabular}
}
\caption{Pictorial examples illustrating the geometrical intuitions behind Algorithm~\ref{alg:certificate1}.}
\label{fig:certificate1}
\end{figure*}

\subsubsection{\textbf{Smaller Conflicting Sensor Set}}
To generate a compact Boolean constraint that explains a conflict, we aim to find a small set of sensors that cannot all be attack-free. Their existence is guaranteed by the following proposition whose proof exploits the geometric interpretation provided by the hyperplanes $\mathbb{H}_i$.

\begin{proposition}
Let the linear dynamical system $\Sigma_a$ defined in~\eqref{eq:system_attack} be $2\overline{s}$-sparse observable. Let  $\norm{\Psi_i} = 0$ for all $i \in \{1,\hdots,p\}$ and let also  $\epsilon = 0$ be the numerical solver tolerance for Algorithm \ref{alg:check}. 
If $\mathcal{T}$\textsc{-Solve.Check}$(\mathcal{I})$ is \texttt{UNSAT} for a set $\mathcal{I}$, with $\vert \mathcal{I} \vert > p - 2\overline{s}$,
then there exists a subset $\mathcal{I}_{temp} \subset \mathcal{I}$ with  $\vert \mathcal{I}_{temp} \vert \le p - 2\overline{s} + 1$ such that $\mathcal{T}$\textsc{-Solve.Check}$(\mathcal{I}_{temp})$ is also \texttt{UNSAT}.
\label{prop:minSet}
\end{proposition}
\begin{proof}
Consider any set of sensors $\mathcal{I}' \subset \mathcal{I}$ such that \mbox{$\vert \mathcal{I}' \vert = p - 2\overline{s}$} and $\bigcap_{i \in \mathcal{I}'} \mathbb{H}_i$ is not empty. If such set $\mathcal{I}'$ does not exist, then the result follows trivially. If the set $\mathcal{I}' $ exists, then it follows from Proposition \ref{cor:injectivity} that $\O_{\mathcal{I}'}$ has a trivial kernel and hence the intersection $\bigcap_{i \in \mathcal{I}'} \mathbb{H}_i$ is a single point, named $x'$. 
Now, since  $\mathcal{T}$\textsc{-Solve.Check}$(\mathcal{I})$  is \texttt{UNSAT},  it follows from Proposition~\ref{prop:geom} that:
\begin{align*}
\bigcap_{i \in \mathcal{I}} \mathbb{H}_i = \emptyset &\Rightarrow 
\bigcap_{i \in \mathcal{I}'} \mathbb{H}_i \cap \bigcap_{i \in \mathcal{I}\setminus\mathcal{I}'} \mathbb{H}_i  = \emptyset \Rightarrow 
\{x'\} \cap \bigcap_{i \in \mathcal{I}\setminus\mathcal{I}'} \mathbb{H}_i  = \emptyset,
\end{align*}
which in turn implies that there exists at least one sensor $i \in \mathcal{I}\setminus\mathcal{I}'$ such that its hyperplane $\mathbb{H}_i$ does not pass through the point $x'$. Now, we define $\mathcal{I}_{temp}$ as $\mathcal{I}_{temp} = \mathcal{I}' \cup i$ and note that $\vert \mathcal{I}_{temp} \vert \le p - 2\overline{s} + 1$, 
which concludes the proof.
\end{proof}

Using Proposition \ref{prop:minSet}, our objective is to find a small set of hyperplanes that fails to intersect. Hence, if an assignment for the convex constraints is \texttt{UNSAT}, our conjecture is that the $p - 2\overline{s}$ hyperplanes with the lowest (normalized) residuals are most likely to have a common intersection point, which can then be used as a candidate intersection point  for the hyperplanes with the higher (normalized) residuals, one-by-one, until a conflict is detected. A pictorial illustration of this intuition is given in Figure~\ref{fig:heur1}. Based on this intuition, we propose the following procedure, summarized in Algorithm~\ref{alg:certificate1}. 

\begin{algorithm}
\caption{$\mathcal{T}\textsc{-Solve.Certificate-Conflict}({\mathcal{I}}, x)$}
\begin{algorithmic}[1]
\STATE \textbf{Compute normalized residuals} 
\STATE $\quad r := \bigcup_{i \in  \mathcal{I}} \left \{r_i \right \}, \quad r_i := \norm{Y_i - \O_i x }^2/\norm{\O_i}^2, \; i \in {\mathcal{I}}$;
\STATE \textbf{Sort the residual variables} 
\STATE $\quad r\_sorted := \text{sortAscendingly}(r)$;
\STATE \textbf{Pick the index corresponding to the maximum residual} 
\STATE $\quad \mathcal{I}\_max\_r :=  \text{Index}(r\_sorted_{\{\vert \mathcal{I} \vert, \vert \mathcal{I} \vert - 1,\ldots,  p - 2\overline{s} + 1\}})$;
\STATE $\quad \mathcal{I}\_min\_r :=  \text{Index}(r\_sorted_{\{1, \ldots, p - 2\overline{s} \} })$;
\STATE \textbf{Search linearly for the UNSAT certificate}
\STATE $\quad \text{status = \texttt{SAT}}; \quad \text{counter} = 1;$
\STATE $\quad \mathcal{I}\_temp :=  \mathcal{I}\_min\_r \cup \mathcal{I}\_max\_r_{counter}$;
\WHILE{status == \texttt{SAT}}
	\STATE (status, $x$) $:=  \mathcal{T}\textsc{-Solve.Check}(\mathcal{I}\_temp)$;
	\IF{status == \texttt{UNSAT}}
		\STATE $\phi_{\text{conf-cert}}:= \sum_{i \in \mathcal{I}\_temp} b_i \ge 1$;
	\ELSE
		\STATE counter := counter  + 1;
		\STATE $\mathcal{I}\_{temp} :=  \mathcal{I}\_min\_r \cup \mathcal{I}\_max\_r_{counter}$;
	\ENDIF
\ENDWHILE
\STATE \textbf{[Optional] Sort the rest according to $\text{dim}(ker\{\O\})$} 
\STATE$\quad \mathcal{I}\_temp2 = \text{sortAscendingly}(dim(ker\{\O_{\mathcal{I}\_temp}\}))$;
\STATE $\quad \text{status = \texttt{UNSAT}}; \quad \text{counter2} = \vert\mathcal{I}\_temp2 \vert\ - 1;$
\STATE $\quad \mathcal{I}\_temp2 :=  \mathcal{I}\_temp2_{\{1, \dots, counter2\}}$;
\WHILE{status == \texttt{UNSAT}}
	\STATE (status, $x$) $:=  \mathcal{T}\textsc{-Solve.Check}(\mathcal{I}_{temp})$;
	\IF{status == \texttt{SAT}}
		\STATE $\phi_{\text{conf-cert}}:= \sum_{i \in \mathcal{I}\_{temp2}_{\{1, \dots, counter2 + 1\}}} b_i \ge 1$;
	\ELSE
		\STATE counter2 := counter2  - 1;
		\STATE $\mathcal{I}\_temp2 :=  \mathcal{I}\_temp2_{\{1, \dots, counter2\}}$;
	\ENDIF
\ENDWHILE
\STATE \textbf{return} $\phi_{\text{conf-cert}}$
\end{algorithmic}
\label{alg:certificate1}
\end{algorithm}

We first compute the (normalized) residuals $r_i$ for all $i \in \mathcal{I}$, and sort them in ascending order. We then pick the $p - 2\overline{s}$ minimum  (normalized) residuals indexed by $\mathcal{I}\_min\_r$,
and search for one more hyperplane that leads to a conflict with the hyperplanes indexed by $\mathcal{I}\_min\_r$.
To do this, we start by solving the same optimization problem as in Algorithm~\ref{alg:check}, but on the reduced set of hyperplanes indexed by $\mathcal{I}_{temp} = \mathcal{I}\_min\_r\cup \mathcal{I}\_max\_r$, where $\mathcal{I}\_max\_r$ is the index associated with the hyperplane having the maximal (normalized) residual. If this set of hyperplanes intersect in one point, they are labelled as ``non-conflicting'', and we repeat the same process by replacing the hyperplane indexed by $\mathcal{I}\_max\_r$ with the hyperplane associated with the second maximal (normalized) residual from the sorted list, till we reach a conflicting set of hyperplanes. Once the set is discovered, we stop by generating the following, more compact, certificate:
$$ \phi_{\text{conf-cert}}:= \sum_{i \in \mathcal{I}_{temp}} b_i \ge 1.$$
A sample execution of Algorithm~\ref{alg:certificate1} is illustrated in Figure~\ref{fig:timelapse1}.


Finally, as a post-processing step, we can further reduce the cardinality of $\mathcal{I}_{temp}$ by exploiting the dimension of the hyperplanes corresponding to the index list. Intuitively, the lower the dimension, the more information is provided by the corresponding sensor. For example, a sensor $i$ with \mbox{$\text{dim}(\mathbb{H}_i) = \text{dim}(\ker \O_i) = 0$} corresponds to only one point $\O_i^{-1} Y_i$. This restricts the search space to the unique point and makes it easier to generate a conflict formula. Therefore, to converge faster towards a conflict, we iterate through the indexes in $\mathcal{I}_{temp}$ and remove at each step the one which corresponds to the hyperplane with the highest dimension until we are left with a reduced index set that is still conflicting.

The following result  provides an upper bound for the performance of the proposed heuristic.
\begin{proposition}
Let the linear dynamical system $\Sigma_a$ defined in~\eqref{eq:system_attack} be $2\overline{s}$-sparse observable. Let  $\norm{\Psi_i} = 0$ for all $i \in \{1,\hdots,p\}$ and let also 
$\epsilon = 0$ be the numerical solver tolerance for Algorithm \ref{alg:check}. Then, Algorithm~\ref{alg:smt} using the conflicting UNSAT certificate $\phi_{\text{conf-cert}}$ in Algorithm~\ref{alg:certificate1} is $\delta$-complete (in the sense of Definition \ref{def:deltacom}) with $\delta = 0$.
Moreover, the upper bound on the number of iterations of Algorithm~\ref{alg:smt} is 
$\binom{p}{p - 2\overline{s}+1}$.
\label{prop:certifcate1}
\end{proposition}

\begin{proof}
$\delta$-Completeness follows from Lemma \ref{prop:check} along with the $2\overline{s}$ observability condition. The upper bound on the number of iterations of Algorithm \ref{alg:smt} can be derived as follows. First, it follows from Proposition \ref{prop:minSet} that each certificate $\phi_{\text{conf-cert}}$ has at most $p - 2\overline{s} + 1$ sensors. Since we know that the algorithm always terminates, the worst case would then happen when the solver exhaustively generates all conflicting sets of cardinality $p - 2\overline{s} + 1$. This leads to a number of iterations equal to
$\binom{p}{p - 2\overline{s}+1}.$

\end{proof}


\subsubsection{\textbf{Agreeable Sensor Set}}
This heuristic aims to find a set of $p - 2\overline{s}$ sensors that all agree on the same $x$. We recall that the $2\overline{s}$-sparse observability condition ensures that the state is fully observable from any set of $p - 2\overline{s}$ sensors. Accordingly, for a given set of sensors, we select the $p - 2\overline{s}$ sensors, hence hyperplanes, that correspond to minimal residuals. We then check whether they all intersect in one point $x$. In such case, we inform the SAT solver that all of these sensors are unattacked, by generating the following certificate:
$$ \phi_{\text{agree-cert}}:= \sum_{i \in \mathcal{I}\_min\_r} b_i = 0, $$ 
where $\mathcal{I}\_min\_r$ is the set of indexes of the  $p - 2\overline{s}$ hyperplanes with the lowest residuals. 

The procedure described above is summarized in Algorithm~\ref{alg:certificate2}. As evident from line 9 of Algorithm~\ref{alg:certificate2},  $\phi_{\text{agree-cert}}$ is not always generated; therefore, we use this heuristic, when it is successful, only as a complement of the previously discussed UNSAT certificate. Moreover, the heuristic itself is not always applicable. In fact, it is still possible to design an attack such that up to $\overline{s}$ attacked sensors agree on a single value of $x$. Hence, an additional condition is required as reflected in the following proposition. 

\begin{proposition}
Let the linear dynamical system $\Sigma_a$ defined in~\eqref{eq:system_attack} be $3\overline{s}$-sparse observable. Let  $\norm{\Psi_i} = 0$ for all $i \in \{1,\hdots,p\}$ and let also
$\epsilon = 0$ be the numerical solver tolerance for Algorithm \ref{alg:check}. Then, Algorithm~\ref{alg:smt} using the agreeable UNSAT certificate $\phi_{\text{agree-cert}}$ in Algorithm~\ref{alg:certificate2} is $\delta$-complete (in the sense of Definition \ref{def:deltacom}) with $\delta = 0$.
Moreover, whenever $\phi_{\text{agree-cert}}$ is generated,  Algorithm \ref{alg:smt} terminates within $\sum_{s = 0}^{\overline{s}} \binom{2\overline{s}}{s}$ iterations.
\end{proposition}

\begin{proof}
$\delta$-Completeness of Algorithm~\ref{alg:certificate2} is equivalent to showing the soundness and completeness of Algorithm~\ref{alg:check}. It follows from Proposition~\ref{prop:check} that Algorithm~\ref{alg:check} is sound and complete whenever the system is $2 \overline{s}$-sparse observable and when the cardinality of  $\mathcal{I}$ satisfies $\vert \mathcal{I} \vert \ge p - \overline{s}$. Hence, to show the result, it is enough to replicate the proof of Proposition~\ref{prop:check} under the assumption that the system is $3\overline{s}$-sparse observable and the cardinality of $\mathcal{I}$ satisfies instead $\vert \mathcal{I} \vert \ge p - 2\overline{s}$. 

The bound on the number of iterations can be derived as follows. First, we note that $\phi_{\text{agree-cert}}$ assigns $p - 2\overline{s}$ as being unattacked sensors. This in turn forces the solver to search for the attacked sensors in the remaining set of sensors with caridinality $p - (p - 2\overline{s}) = 2\overline{s}$. The bound then follows using the same argument of Proposition~\ref{prop:trivial-cert}.
\end{proof}

\begin{algorithm}[ht]
\caption{$\mathcal{T}\textsc{-Solve.Certificate-Agree}(\mathcal{I}, x)$}
\begin{algorithmic}[1]
\STATE \textbf{Compute individual residuals} 
\STATE $\quad r := \bigcup_{i \in  \mathcal{I}}\left \{r_i \right \}, \quad r_i := \norm{Y_i - \O_i x}^2/\norm{\O_i}^2, \; i \in {\mathcal{I}}$;
\STATE \textbf{Sort the residual variables} 
\STATE $\quad r\_sorted := \text{sortAscendingly}(r)$;
\STATE \textbf{Pick the $p - 2\overline{s}$ indexes corresponding to the minimum residuals} 
\STATE $\quad \mathcal{I}\_min\_r :=  \text{Index}(r\_sorted_{\left \{1, \ldots, p - 2\overline{s} \right \} }))$;
\STATE $\quad (\text{status},x ) :=  \mathcal{T}\textsc{-Solve.Check}(\mathcal{I}\_min\_r)$;
\STATE $\phi_{\text{agree-cert}}:= $\texttt{TRUE};
\IF{status == \texttt{SAT}}
	\STATE $\phi_{\text{agree-cert}}:= \sum_{i \in \mathcal{I}\_min\_r} b_i = 0$;
\ENDIF
\STATE \textbf{return} $\phi_{\text{agree-cert}}$
\end{algorithmic}
\label{alg:certificate2}
\end{algorithm}

\subsection{Soundness and Completeness of Algorithm \ref{alg:smt}, Noiseless Case}
The procedure $\mathcal{T}\textsc{-Solve.Certificate}(\mathcal{I}, x)$ in line  7 of Algorithm~\ref{alg:smt} can be implemented as shown in Algorithm~\ref{alg:certificate}. We are now ready to state the main result of this section, which is a direct consequence of our previous results.

\begin{theorem}
Let the linear dynamical system $\Sigma_a$ defined in~\eqref{eq:system_attack} be $2\overline{s}$-sparse observable. $\norm{\Psi_i} = 0$ for all $i \in \{1,\hdots,p\}$ and let also $\epsilon=0$ be the numerical solver tolerance for Algorithm \ref{alg:check}.  Algorithm~\ref{alg:smt} is $\delta$-complete (in the sense of Definition \ref{def:deltacom}) with $\delta=0$.
\label{th:noiselss}
\end{theorem}

\begin{algorithm}
\caption{$\mathcal{T}\textsc{-Solve.Certificate}(\mathcal{I}, x)$}
\begin{algorithmic}[1]
\STATE $\phi_{\text{cert}} := \mathcal{T}\textsc{-Solve.Certificate-Conflict}({\mathcal{I}}, x)$;
\IF{$p > 3 \overline{s}$}
	\STATE $\phi_{\text{agree-cert}}:= \mathcal{T}\textsc{-Solve.Certificate-Agree}(\mathcal{I}, x)$;
	\STATE $\phi_{\text{cert}} := \phi_{\text{cert}} \land \phi_{\text{agree-cert}}$;
\ENDIF
\STATE \textbf{return} $\phi_{\text{cert}}$
\end{algorithmic}
\label{alg:certificate}
\end{algorithm}

\section{Completeness in the Presence of Noise}
\label{sec:comp}


As discussed in the previous section, \textsc{Imhotep-SMT} can always detect any compromised sensors in the absence of measurement noise ($\norm{\Psi_i} = 0$ for all $i \in \{1,\hdots,p\}$) and when the numerical tolerance is zero ($\epsilon = 0$). In this section, we characterize completeness in the presence of noise and/or numerical tolerance in the solver, by determining to what extent an attack signal can be hidden by noise and/or numerical tolerance, thereby making it unfeasible to reconstruct the true state. Since Algorithm~\ref{alg:smt} consists of multiple invocations of the least-squares problem, the completeness of the detector entirely depends on the correctness of Algorithm~\ref{alg:check} in checking the satisfiability of a Boolean assignment over $b$.  

The completeness of Algorithm~\ref{alg:check} will in turn depend on two major components: (i) the tolerance of the numerical solvers, which is typically  a small value used as a stopping criterion, and can be controlled by the user; (ii) the noise margin intrinsic to the dynamical system model.
To account for these two components, we replace the satisfiability condition in line~\ref{line:check} of Algorithm~\ref{alg:check} with the following condition:
\begin{equation} \label{eq:noise}
\norm{Y_{\mathcal{I}} - \O_{\mathcal{I}}x} \le \norm{\Psi_{\mathcal{I}}} + \epsilon
\end{equation}
\noindent where $\epsilon > 0$ is the user-defined tolerance. To characterize soundness and completeness of Algorithm~\ref{alg:check}, we first recall that the solution of the unconstrained least squares problem in Algorithm~\ref{alg:check} is given by:
\begin{align*}
x 	&= \left(\O_{\mathcal{I}}^T \O_{\mathcal{I}}\right)^{-1} \O_{\mathcal{I}}^T Y_{\mathcal{I}}
	= \O_{\mathcal{I}}^{+} Y_{\mathcal{I}}
\end{align*}
where $\O_{\mathcal{I}}^{+} = \left(\O_{\mathcal{I}}^T \O_{\mathcal{I}}\right)^{-1} \O_{\mathcal{I}}^T$ is the Moore-Penrose pseudo inverse of $\O_{\mathcal{I}}$. It is apparent that  soundness and completeness of Algorithm~\ref{alg:check} depends on the properties of the matrix $\O_{\mathcal{I}}^{+}$. Accordingly, we define the following two, technical, quantities.
\begin{definition}
Define $\overline{o} \in \R^+$ as:
$$ \overline{o} = \max_{\mathcal{I} \subseteq \{1, \ldots,p\}} \norm{\O_{\mathcal{I}}^{+}}^2$$
where $\O_{\mathcal{I}}^{+}$ Moore-Penrose pseudo inverse of $\O_{\mathcal{I}}$.
\label{def:o_bar}
\end{definition}

\begin{defProp}
Let the linear system defined in~\eqref{eq:system_attack} be $2\overline{s}$-sparse observable and
define $\Delta_{s} \in \R^{+}$ as:
\begin{align*}
\Delta_{s} &=\max_{\substack{\Gamma \subset \mathcal{I} \subseteq \{1,\ldots,p\}\\ \vert \Gamma \vert \le \overline{s}, \vert \mathcal{I} \vert \ge p - \overline{s}  }} \lambda_{\max}\left\{ \left(\sum_{i \in \Gamma}\O_i^T \O_i \right) \left(\sum_{i \in \mathcal{I}}\O_i^T \O_i\right)^{-1} \right\}
\end{align*}
Then, for any 
$\overline{s}$-sparse attack vector $E$, the following holds:
$$ \norm{(I - \O_{\mathcal{I}}\O_{\mathcal{I}}^{+})E_{\mathcal{I}}}^2 \ge (1 - \Delta_{s}) \norm{E_{\mathcal{I}}}^2$$
with $\Delta_{s}$ strictly less than 1.
\label{prop:appendix1}
\end{defProp}
\begin{proof}
First define the set $\Gamma \subset \mathcal{I}$ as the set of indices on which the attack vector $E$ is supported and note that $E_{\overline{\Gamma}} = 0$.
Hence:
\begin{align*}
\norm{(I - \O_{\mathcal{I}}\O_{\mathcal{I}}^{+})E_{\mathcal{I}}}^2 
&= E_{\mathcal{I}}^T \left(I - \O_{\mathcal{I}}\O_{\mathcal{I}}^{+}\right)^2 E_{\mathcal{I}} 
\stackrel{(a)}{=} E_{\mathcal{I}}^T \left(I - \O_{\mathcal{I}}\O_{\mathcal{I}}^{+}\right) E_{\mathcal{I}} \\
&= E_{\mathcal{I}}^TE_{\mathcal{I}} - E_{\mathcal{I}}^T \O_{\mathcal{I}} (\O_{\mathcal{I}}^T \O_{\mathcal{I}})^{-1} \O_{\mathcal{I}}^T E_{\mathcal{I}} 
\stackrel{(b)}{=} E_{\Gamma}^TE_{\Gamma} - E_{\Gamma}^T \O_{\Gamma} (\O_{\mathcal{I}}^T \O_{\mathcal{I}})^{-1} \O_{\Gamma}^T E_{\Gamma}
\end{align*}
where, equality $(a)$ follows from the fact that the matrix $I - \O_{\mathcal{I}}\O_{\mathcal{I}}^{+}$ is idempotent and equality $(b)$ follows from the definition of the set $\Gamma$. The second term can be bounded as:
\begin{align*}
E_{\Gamma}^T \O_{\Gamma} &(\O_{\mathcal{I}}^T \O_{\mathcal{I}})^{-1} \O_{\Gamma}^T E_{\Gamma}
\le \lambda_{\max}\{\O_{\Gamma} (\O_{\mathcal{I}}^T \O_{\mathcal{I}})^{-1} \O_{\Gamma}^T\} E_{\Gamma}^T E_{\Gamma}
\end{align*}
Hence, to show that the result holds, we need to show that the inequality:
\begin{align}
 \lambda_{\max}\{\O_{\Gamma} (\O_{\mathcal{I}}^T \O_{\mathcal{I}})^{-1} \O_{\Gamma}^T\} < 1
 \label{eq:toBeShown}
 \end{align}
holds for any set $\mathcal{I}$ and $\Gamma \subset \mathcal{I}$ with $\vert\Gamma\vert \le \overline{s}$ and $\vert\mathcal{I}\vert \ge p - \overline{s}$. First, recall that for any two matrices $A$ and $B$ with appropriate dimensions, $\lambda_{\max}\{AB\} = \lambda_{\max}\{BA\}$. Hence, we can rewrite \eqref{eq:toBeShown} as:
$$\lambda_{\max}\{\O_{\Gamma}^T\O_{\Gamma} (\O_{\mathcal{I}}^T \O_{\mathcal{I}})^{-1} \} < 1.
 $$
Now notice that: 
\begin{align*}
\O_{\mathcal{I}}^T \O_{\mathcal{I}} 
	&= \sum_{i \in \mathcal{I}} \O_i^T\O_i 
	= \sum_{i \in \Gamma} \O_i^T\O_i +  \sum_{i \in \mathcal{I}\setminus\Gamma} \O_i^T\O_i \\
	&=\O_{\Gamma}^T \O_{\Gamma} + \O_{\mathcal{I}\setminus\Gamma}^T \O_{\mathcal{I}\setminus\Gamma}
\end{align*}
and rewrite \eqref{eq:toBeShown} as:
$$ \lambda_{\max}\left\{\O_{\Gamma}^T\O_{\Gamma} \left(\O_{\Gamma}^T \O_{\Gamma} + \O_{\mathcal{I}\setminus\Gamma}^T \O_{\mathcal{I}\setminus\Gamma} \right)^{-1} \right\} < 1$$
where the set $\mathcal{I}\setminus\Gamma$ has a cardinality of at least $p - 2\overline{s}$. Hence, it follows from the $2s$-sparse observability condition that the matrix $\O_{\mathcal{I}\setminus\Gamma}^T \O_{\mathcal{I}\setminus\Gamma}$ is positive definite and therefore we can apply Proposition \ref{prop:helper} in the appendix to show that the statement holds.
\end{proof}

Using these two quantities, we can state our main result, which is the noisy version of Theorem \ref{th:noiselss}.
\begin{theorem}
Let the linear system defined in~\eqref{eq:system_attack} be $2\overline{s}$-sparse observable, and let $\epsilon > 0$ be the numerical solver tolerance. Then, for any attack $E_i$ satisfying:
\begin{align}
\norm{E_i}^2 > \left(\frac{2}{1 - \Delta_{s}} \right) \norm{\Psi}^2 + \frac{\epsilon}{1 - \Delta_{s}},
\label{eq:Enorm}
\end{align}
Algorithm~\ref{alg:smt}, modified as in~\eqref{eq:noise}, is $\delta$-complete with \mbox{$\delta = \overline{o} \norm{\Psi}^2$}.
\label{th:smt}
\end{theorem}

\begin{proof}
To prove the result, we need to show that the check~\eqref{eq:noise}, resulting in $\delta$-satisfiability, is satisfied if and only if no sensor in $\mathcal{I}$ is under attack.

If no sensor is under attack, condition~\eqref{eq:noise} is trivially satisfied. Therefore, we focus on proving the reverse implication, showing that if at least one sensor $i_a \in \mathcal{I}$ is under attack, then~\eqref{eq:noise} does not hold as long as the attack $E_{i_a}$ satisfies~\eqref{eq:Enorm}.
To do so, we consider the set $\mathcal{I}$ which contains the attacked sensor $i_a$. Recall that the solution of the unconstrained least squares problem in Algorithm~\ref{alg:check} is given by:
\begin{align*}
x 	&= \left(\O_{\mathcal{I}}^T \O_{\mathcal{I}}\right)^{-1} \O_{\mathcal{I}}^T Y_{\mathcal{I}}
	= \O_{\mathcal{I}}^{+} Y_{\mathcal{I}}
\end{align*}
where $\O_{\mathcal{I}}^{+} = \left(\O_{\mathcal{I}}^T \O_{\mathcal{I}}\right)^{-1} \O_{\mathcal{I}}^T$ is the Moore-Penrose pseudo inverse of $\O_{\mathcal{I}}$. Hence, the value of the objective function at the optimal point can be bounded from below as:
\begin{align}
\norm{Y_{\mathcal{I}} - \O_{\mathcal{I}} x }^2
	&\stackrel{(a)}{=} \norm{Y_{\mathcal{I}} - \O_{\mathcal{I}} \O_{\mathcal{I}}^{+} Y_{\mathcal{I}} }^2 \nonumber 
	\stackrel{(b)}{=} \norm{ (I - \O_{\mathcal{I}}\O_{\mathcal{I}}^{+} ) (\O_{\mathcal{I}} x^* + \Psi_{\mathcal{I}} + E_{\mathcal{I}})}^2 \nonumber \\
	&\stackrel{(c)}{=} \norm{(I - \O_{\mathcal{I}}\O_{\mathcal{I}}^{+} ) (\Psi_{\mathcal{I}} + E_{\mathcal{I}})}^2 
	\stackrel{(d)}{\ge} \left\vert \norm{(I - \O_{\mathcal{I}}\O_{\mathcal{I}}^{+} ) E_{\mathcal{I}}}^2 - \norm{(I - \O_{\mathcal{I}}\O_{\mathcal{I}}^{+} ) \Psi_{\mathcal{I}}}^2 \right\vert
\label{eq:bound1}
\end{align}
where (a) follows from~\eqref{eq:y_equality}; $(b)$ follows from the definition of $Y_{\mathcal{I}}$; $(c)$ follows from the fact that $\O_{\mathcal{I}}\O_{\mathcal{I}}^{+} \O_{\mathcal{I}} = \O_{\mathcal{I}}$;  and the  inequality  $(d)$  follows from the inverse triangular inequality.

On the other hand, the condition on the attack signal~\eqref{eq:Enorm} implies that:
\begin{align*}
\norm{E_{i_a}}^2 &> \left(\frac{2}{1 - \Delta_{s}} \right) \norm{\Psi}^2 + \frac{\epsilon}{1 - \Delta_s} 
 \ge \left(\frac{2}{1 - \Delta_{s}}\right) \norm{\Psi_{\mathcal{I}}}^2 + \frac{\epsilon}{1 - \Delta_s}
\end{align*}
Hence, by noticing that $\norm{E_{\mathcal{I}}}^2 \ge \norm{E_{i_a}}^2$, we conclude that:
\begin{align}
\norm{E_{\mathcal{I}}}^2 &> \left(\frac{2}{1 - \Delta_{s}}\right) \norm{\Psi_{\mathcal{I}}}^2 + \frac{\epsilon}{1 - \Delta_s} \nonumber \\
&\stackrel{}{\Rightarrow} (1 - \Delta_{s}) \norm{E_{\mathcal{I}}}^2 > \norm{\Psi_{\mathcal{I}}}^2 + \norm{\Psi_{\mathcal{I}}}^2 + \epsilon \nonumber \\
&\stackrel{(e)}{\Rightarrow} (1 - \Delta_{s}) \norm{E_{\mathcal{I}}}^2 > \norm{\Psi_{\mathcal{I}}}^2 + \norm{I - \O_{\mathcal{I}}\O_{\mathcal{I}}^{+}}^2\norm{\Psi_{\mathcal{I}}}^2 + \epsilon \nonumber \\
&\stackrel{(f)}{\Rightarrow} (1 - \Delta_{s}) \norm{E_{\mathcal{I}}}^2 > \norm{\Psi_{\mathcal{I}}}^2 + \norm{(I - \O_{\mathcal{I}}\O_{\mathcal{I}}^{+})\Psi_{\mathcal{I}}}^2 + \epsilon \nonumber \\
&\stackrel{(g)}{\Rightarrow} \norm{(I - \O_{\mathcal{I}}\O_{\mathcal{I}}^{+}) E_{\mathcal{I}}}^2 > \norm{\Psi_{\mathcal{I}}}^2 + \norm{(I - \O_{\mathcal{I}}\O_{\mathcal{I}}^{+})\Psi_{\mathcal{I}}}^2 + \epsilon \nonumber \\
&\stackrel{}{\Rightarrow} \norm{(I - \O_{\mathcal{I}}\O_{\mathcal{I}}^{+}) E_{\mathcal{I}}}^2 -  \norm{(I - \O_{\mathcal{I}}\O_{\mathcal{I}}^{+})\Psi_{\mathcal{I}}}^2  > \norm{\Psi_{\mathcal{I}}}^2 + \epsilon \nonumber \\
&\stackrel{(h)}{\Rightarrow} \left\vert \norm{(I - \O_{\mathcal{I}}\O_{\mathcal{I}}^{+}) E_{\mathcal{I}}}^2 -  \norm{(I - \O_{\mathcal{I}}\O_{\mathcal{I}}^{+})\Psi_{\mathcal{I}}}^2  \right\vert > \norm{\Psi_{\mathcal{I}}}^2 + \epsilon
\label{eq:bound2}
\end{align}
where, the implication $(e)$ follows from the fact that the matrix $I - \O_{\mathcal{I}}\O_{\mathcal{I}}^{+}$ is idempotent and hence $\norm{I - \O_{\mathcal{I}}\O_{\mathcal{I}}^{+}}^2 \le 1$; $(f)$ follows from the properties of the induced norm which implies that $\norm{(I - \O_{\mathcal{I}}\O_{\mathcal{I}}^{+})\Psi_{\mathcal{I}}} \le \norm{I - \O_{\mathcal{I}}\O_{\mathcal{I}}^{+}}\norm{\Psi_{\mathcal{I}}}$; $(g)$ follows from Proposition \ref{prop:appendix1}; finally, $(h)$ follows from the right hand side of the inequality being positive and hence the left hand side is also positive along with the fact that $\vert a \vert = a$ whenever $a \in \R$ is positive.

Combining the bounds~\eqref{eq:bound1} and~\eqref{eq:bound2} we conclude that the following holds:
$$\norm{Y_{\mathcal{I}} - \O_{\mathcal{I}} x }^2 >  \norm{\Psi_{\mathcal{I}}}^2 + \epsilon$$
which implies that the result of Algorithm \ref{alg:check} is UNSAT whenever~\eqref{eq:Enorm} is satisfied.

The error bound $\delta$ can be then computed directly as:
\begin{align*}
\norm{x^* - x}^2 	&= \norm{x^* - \O_{\mathcal{I}}^{+} Y_{\mathcal{I}}}^2 
			\stackrel{(i)}{=} \norm{\O_{\mathcal{I}}^{+} \Psi_{\mathcal{I}} }^2 
			\le \norm{\O_{\mathcal{I}}^{+}}^2 \norm{\Psi_{\mathcal{I}}}^2 
			\stackrel{(j)}{\le} \overline{o} \norm{\Psi}^2,
\end{align*}
where, the equality $(i)$ follows from the fact that all attacks satisfy \eqref{eq:Enorm} and hence can be detected. Accordingly, the set $\mathcal{I}$ contains only sensors which is attack free and therefore~\eqref{eq:y_equality} can be simplified into: 
$$Y_{\mathcal{I}} = \O_{\mathcal{I}}x^* + \Psi_{\mathcal{I}}.$$ 
Finally, the inequality $(j)$ the follows from the definition of $\overline{o}$ in~\ref{def:o_bar}.
\end{proof}
\begin{remark}
In the previous proof, we rely on the assumption that:
$$ \norm{E_{\mathcal{I}}}^2 > \left(\frac{2}{1 - \Delta_{s}}\right) \norm{\Psi_{\mathcal{I}}}^2 + \frac{\epsilon}{1 - \Delta_s} \nonumber$$
However, since we do not know the set $\mathcal{I}$, which is selected by the underlying SAT solver, we resort to the more conservative assumption:
$$ \norm{E_i}^2 > \left(\frac{2}{1 - \Delta_{s}} \right) \norm{\Psi}^2 + \frac{\epsilon}{1 - \Delta_{s}}$$
that will be used in Theorem \ref{th:noisy2}.
\label{rem:relaxation}
\end{remark}

The previous result characterizes the class of attack signals that will lead to detection. However, a smart attacker may be tempted to inject attack signals which are not detected by the proposed algorithm, but yet increase the estimation error. The following result, characterizes the estimation error in the presence of un-detectable attacks.
\begin{theorem}
Let the linear system defined in~\eqref{eq:system_attack} be $2\overline{s}$-sparse observable, and let $\epsilon > 0$ be the numerical solver tolerance. Algorithm~\ref{alg:smt}, modified as in~\eqref{eq:noise}, returns an estimate $x$ which satisfies:
$$\norm{x^* - x}^2 \le 2\overline{o} \left(1 + \frac{2}{1 - \Delta_s}\right) \norm{\Psi}^2 +  \frac{ 2\overline{o} \epsilon}{1 - \Delta_s}.$$
\label{th:noisy2}
\end{theorem}
\begin{proof}
The error $\norm{x^* - x}^2$ can be bounded as follows:
\begin{align*}
\norm{x^* - x}^2 	&= \norm{x^* - \O_{\mathcal{I}}^{+} Y_{\mathcal{I}}}^2 
	= \norm{x^* - \O_{\mathcal{I}}^{+} \O_{\mathcal{I}} x^* - \O_{\mathcal{I}}^{+} \Psi_{\mathcal{I}} - \O_{\mathcal{I}}^{+} E_{\mathcal{I}}}^2 
			= \norm{\O_{\mathcal{I}}^{+} \Psi_{\mathcal{I}} - \O_{\mathcal{I}}^{+} E_{\mathcal{I}}}^2 \\
			&\stackrel{(a)}{\le} 2\norm{\O_{\mathcal{I}}^{+}}^2 \norm{\Psi_{\mathcal{I}}}^2 + 2\norm{\O_{\mathcal{I}}^{+}}^2 \norm{E_{\mathcal{I}}}^2\\
			&\stackrel{(b)}{\le} 2\overline{o } \norm{\Psi}^2 + 2\overline{o } \norm{E_{\mathcal{I}}}^2\\
			&\stackrel{(c)}{\le} 2\overline{o} \norm{\Psi}^2 + 2\overline{o} \frac{2}{1 - \Delta_s}\norm{\Psi}^2 + 2\overline{o} \frac{\epsilon}{1 - \Delta_s}\\			
			&= 2\overline{o} \left(1 + \frac{2}{1 - \Delta_s}\right) \norm{\Psi}^2 +  \frac{ 2\overline{o} \epsilon}{1 - \Delta_s}			
\end{align*}
where the inequality $(a)$ follows from Cauchy-Schwarz inequality; $(b)$ from the definition of $\overline{o}$ in~\eqref{def:o_bar}  along with the fact that $\norm{\Psi_{\mathcal{I}}}^2 \le \norm{\Psi}^2$; finally $(c)$ follows from Theorem \ref{th:smt} (along with Remark \ref{rem:relaxation}) which shows that 
only attacks  with norm:
$$\norm{E_\mathcal{I}}^2 \le \left(\frac{2}{1 - \Delta_{s}} \right) \norm{\Psi}^2 + \frac{\epsilon}{1 - \Delta_{s}}$$
are not detected by Algorithm \ref{alg:smt} and hence can affect the estimation error.
\end{proof}
\section{Experimental Results}
\label{sec:results}

We developed our theory solver in \textsc{Matlab}, and interfaced it with the pseudo-Boolean SAT solver \textsc{SAT4J}~\cite{j4sat}. All the experiments were executed on an Intel Core i7 3.4-GHz processor with 8~GB of memory.
To validate our approach, we first compare the effect of the two proposed heuristics on the required number of iterations. We then compare the runtime performance against previously proposed algorithms. Then, we demonstrate the effect of attack detection on the problem of controlling a robotic vehicle under sensor attacks. 

\subsection{Runtime Performance}
To assess the effectiveness of the heuristics introduced in Sec.~\ref{sec:heur}, Figure~\ref{fig:heuristic} shows the number of iterations of \textsc{Imhotep-SMT} when only one of the three certificates, the trivial certificate $\phi_{\text{triv-cert}}$, the conflicting certificate $\phi_{\text{conf-cert}}$, and the joint certificate  $\phi_{\text{conf-cert}}\land\phi_{\text{agree-cert}}$, is used.

In the first experiment (top), we increase the number $s$ of actual sensors under attack for a fixed $\overline{s}= 20$ ($n=25$, $p = 60$). In the second experiment (bottom), we increase both $n$ and $p$ simultaneously, with $p = 3n$, while $p/3$ sensors are under attack, and $\overline{s}= p/3$. In both cases,
the system is constructed to be $3\overline{s}$-sparse observable, with the dimensions of the kernels of $\O_i$ ranging between $n - 1$ and $n - 2$, meaning that the state is ``poorly'' observable from individual sensors. We also show the number of iterations against the theoretical limit in Proposition~\ref{prop:certifcate1}. We observed an average of $50 \times$ reduction in iterations when $\phi_{\text{conf-cert}}$ was used compared to $\phi_{\text{triv-cert}}$, while using both $\phi_{\text{conf-cert}}$ and $\phi_{\text{agree-cert}}$ decreased the number of iterations by a factor of $75$.

We also compared the performance of \textsc{Imhotep-SMT} against the MIQP formulation~\eqref{eq:miqp}, the ETPG algorithm~\cite{YasserETPGarXiv}, and the $l_r/l_1$ decoder~\cite{Hamza_TAC}, with respect to both execution time and estimation error.

The MIQP is solved using the commercial solver \textsc{Gurobi}~\cite{gurobi}, the ETPG algorithm is implemented in \textsc{Matlab},
while the $l_r/l_1$ decoder is implemented using the convex solver \textsc{CVX}~\cite{cvx}. 
Figure \ref{fig:certificate1} reports the numerical results in two test cases. In Figure~\ref{fig:test1}, we fix the number of sensors $p= 20$ and increase the number of system states from $n = 10$ to $n = 150$. In Figure~\ref{fig:test2}, we fix the number of states $n = 50$ and increase the number of sensors from $p=3$ to $p=150$. In both cases, half of the sensors are attacked. 
Our algorithm always outperforms both the ETPG and the $l_r/l_1$ approaches and scales nicely with respect to both $n$ and $p$. In particular, as evident from Figure~\ref{fig:test1}, increasing $n$ has a small effect on the overall execution time, which reflects the fact that the number of constraints to be satisfied does not depend on $n$. Conversely, as shown in Figure~\ref{fig:test2}, as the number of sensors increases, the number of constraints, hence the execution time of our algorithm, also increases. The runtime of the MIQP formulation in (\ref{eq:miqp}) scales worse than our algorithm with $n$, but better with $p$, because \textsc{Gurobi} can efficiently process many conic constraints (whose number scales with $p$) but is more sensitive to the size of each conic constraint (which scales with $n$). 
Finally, Figure~\ref{fig:test1} (bottom) shows that the  $l_r/l_1$ decoder reports incorrect results in multiple test cases, because of its lack of soundness, as discussed in Section~\ref{sec:intro}.

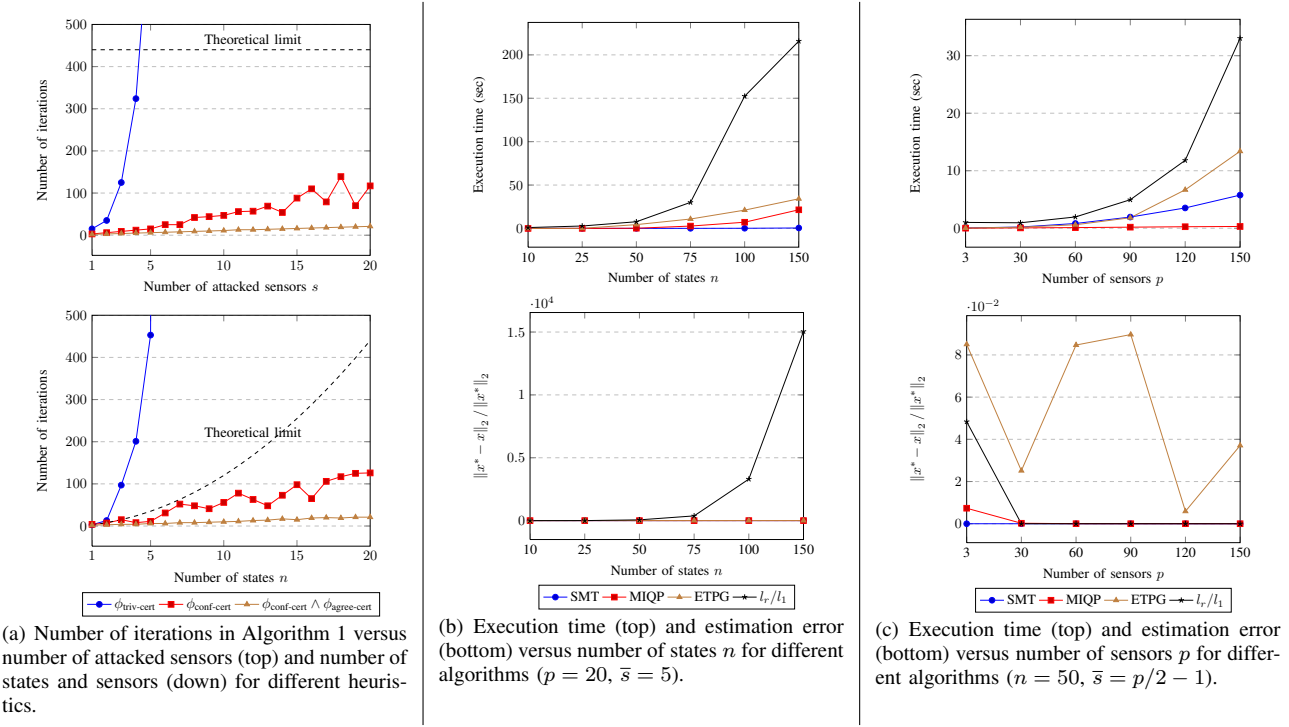
\begin{figure*}[!t]
\centering
{
\begin{tabular}{c|c|c}
	\subfigure[b][Number of iterations in Algorithm~\ref{alg:smt} versus number of attacked sensors (top) and number of states and sensors (down) for different heuristics.
	 ]{\label{fig:heuristic} 
	\begin{tabular}{c}	 
\resizebox{0.26\textwidth}{!}{
	\begin{tikzpicture}
		\begin{axis}[
	    	xlabel=Number of attacked sensors $s$, 
	    	ylabel=Number of iterations,
	    	xtick={1,5,10,15,20},
		ymax=500,
		xmin = 1,
		xmax = 20,
	    	legend style={at={(0.5,-0.2)},anchor=north,legend columns=-1},
	    	ymajorgrids=true,
	    	grid style=dashed,
		]
		\addplot 
			coordinates {(1, 15) (2, 35)(3, 125)(4, 324)(5, 778)(6, 10000)(7, 10000)(8, 10000)(9, 10000)(10, 10000)(11, 10000)(12, 10000)(13, 10000)(14, 10000)(15, 10000)(16, 10000)(17, 10000)(18, 10000)(19, 10000)(20, 10000)};
		\addplot 
			coordinates {(1, 3) (2, 6)(3, 9)(4, 12)(5, 15)(6, 25)(7, 25)(8, 42)(9, 44)(10, 47)(11, 56)(12, 57)(13, 69)(14, 54)(15, 88)(16, 110)(17, 79)(18, 139)(19, 70)(20, 117)};
		\addplot[mark=triangle*,color=brown] 
			coordinates {(1, 2) (2, 3)(3, 4)(4, 5)(5, 6)(6, 7)(7, 8)(8, 9)(9, 10)(10, 11)(11, 13)(12, 13)(13, 14)(14, 15)(15, 16)(16, 17)(17, 18)(18, 19)(19, 20)(20, 21)};
		\addplot  [domain=1:20,dashed,color=black] {440};
		\node [above] at (axis cs:  12,  442) {Theoretical limit};
		\end{axis}
	\end{tikzpicture}}
	\\
	\resizebox{0.26\textwidth}{!}{
	\begin{tikzpicture}
		\begin{axis}[
	    	xlabel=Number of states $n$,
	    	ylabel=Number of iterations,
	    	xtick={1,5,10,15,20},
		ymax=500,
		xmin = 1,
		xmax = 20,
	    	legend style={at={(0.5,-0.2)},anchor=north,legend columns=-1},
	    	ymajorgrids=true,
	    	grid style=dashed,
		]
		\addplot 
			coordinates {(1, 3) (2, 13)(3, 97)(4, 201)(5, 453)(6, 10000)(7, 10000)(8, 10000)(9, 10000)(10, 10000)(11, 10000)(12, 10000)(13, 10000)(14, 10000)(15, 10000)(16, 10000)(17, 10000)(18, 10000)(19, 10000)(20, 10000)};
		\addplot 
			coordinates {(1, 4) (2, 7)(3, 15)(4, 8)(5, 11)(6, 31)(7, 52)(8, 48)(9, 41)(10, 56)(11, 78)(12, 63)(13, 48)(14, 73)(15, 98)(16, 65)(17, 106)(18, 117)(19, 125)(20, 126)};
		\addplot[mark=triangle*,color=brown] 
			coordinates {(1, 2) (2, 3)(3, 4)(4, 4)(5, 6)(6, 6)(7, 8)(8, 8)(9, 9)(10, 10)(11, 11)(12, 13)(13, 14)(14, 17)(15, 15)(16, 19)(17, 20)(18, 19)(19, 21)(20, 21)};
		\addplot  [domain=1:20,dashed,color=black] coordinates {(1, 3) (2, 8)(3, 15)(4, 24)(5, 35)(6, 48)(7, 63) (8, 80)(9, 99)(10, 120)(11, 143)(12, 168)(13, 195)(14, 224)(15, 255)(16, 288)(17, 323)(18, 360)(19, 399)(20, 440)};
		\node [above] at (axis cs:  12,  200) {Theoretical limit};
		\legend{$ \phi_{\text{triv-cert}}$,$\phi_{\text{conf-cert}}$,$\phi_{\text{conf-cert}}\land\phi_{\text{agree-cert}}$}
		\end{axis}
	\end{tikzpicture}}
	\end{tabular}
	}
	&
	\subfigure[b][Execution time (top) and estimation error (bottom) versus number of states $n$ for different algorithms ($p = 20$, $\overline{s} = 5$).
	 ]{\label{fig:test1} 
	\begin{tabular}{c}	 
\resizebox{0.26\textwidth}{!}{
	\begin{tikzpicture}
		\begin{axis}[
	    	xlabel=Number of states $n$,
	    	ylabel=Execution time (sec),
	    	xtick={10,25,50,75,100,150},
		xmin = 10,
		xmax = 150,
		symbolic x coords={10,25,50,75,100,150},
	    	legend style={at={(0.5,-0.2)},anchor=north,legend columns=-1},
	    	ymajorgrids=true,
	    	grid style=dashed,
		]
		\addplot 
			coordinates {(10, 0.150835815) (25, 0.072837016) (50, 0.091151843) (75, 0.128540586) (100, 0.205442241) (150, 0.544454767) };
		\addplot 
			coordinates {(10, 0.01320505142211914) (25, 0.05978202819824219) (50, 0.24718999862670898) (75, 2.7412519454956055) (100, 7.174839019775391) (150, 21.495587825775146) };
		\addplot[mark=triangle*,color=brown] 
			coordinates {(10, 0.054415376) (25, 0.164714279) (50, 4.420559318) (75, 10.85869584) (100, 21.110532127) (150, 34.167959785) }; 
		\addplot 
			coordinates {(10, 1.019596283) (25, 2.8259) (50, 7.7142) (75, 30.048965519) (100, 152.382031462) (150, 215.8006) }; 
		\end{axis}
	\end{tikzpicture} }
	\\
	 \resizebox{0.26\textwidth}{!}{
	\begin{tikzpicture}
		\begin{axis}[
	    	xlabel=Number of states $n$,
	    	ylabel=$\norm{x^* - x}/\norm{x^*}$,
	    	xtick={10,25,50,75,100,150},
		xmin=10,
		xmax=150,
		symbolic x coords={10,25,50,75,100,150},
	    	legend style={at={(0.5,-0.2)},anchor=north,legend columns=-1},
	    	ymajorgrids=true,
	    	grid style=dashed,
		]
		\addplot 
			coordinates {(10, 3.51934918096732e-16) (25, 4.18319789516969e-15) (50, 9.06265435178551e-14) (75, 2.69015619397717e-11) (100, 1.76269257887789e-08) (150, 0.582319563658814e-07) };
		\addplot 
			coordinates {(10,2.381203744599124e-15) (25,2.3172417275056702e-14) (50, 0.25989566923388424) (75, 0.04955397501154004) (100, 0.611309442984122) (150, 0.811047709485256) };
		\addplot[mark=triangle*,color=brown] 
			coordinates {(10, 0.144953900582611) (25, 2.91268173257971) (50, 3.30516279285173) (75, 3.75166123703284 ) (100, 4.56744698319547) (150, 3.29512602048372) }; 
		\addplot 
			coordinates {(10, 1.77374080068667e-08) (25, 1.7949364017653) (50, 59.9213250458421) (75, 374.976671319856) (100, 3300.50968806787) (150, 15047.58231956) }; 
		\legend{SMT,MIQP,ETPG,$l_r/l_1$}
		\end{axis}
	\end{tikzpicture}}
	\end{tabular}
	}&
	\subfigure[ 
	Execution time (top) and estimation error (bottom) versus number of sensors $p$  for different algorithms ($n = 50$, $\overline{s} = p/2-1$).]{\label{fig:test2} 
\begin{tabular}{c}
	\resizebox{0.26\textwidth}{!}{
	\begin{tikzpicture}	
		\begin{axis}[
	    	xlabel=Number of sensors $p$,
	    	ylabel=Execution time (sec),
	    	xtick={3,30,60,90,120,150},
		xmin=3,
		xmax=150,
		symbolic x coords={3,30,60,90,120,150},		
	    	legend style={at={(0.5,-0.2)},anchor=north,legend columns=-1},
	    	ymajorgrids=true,
	    	grid style=dashed,
	]
		\addplot 
			coordinates {(3, 0.0460) (30, 0.2129) (60, 0.8241) (90, 1.9462) (120, 3.5297) (150, 5.7615) };
		\addplot 
			coordinates {(3, 0.012226104736328125) (30, 0.0613710880279541) (60, 0.11451292037963867) (90, 0.1866769790649414) (120, 0.26425719261169434) (150, 0.308596134185791) };
		\addplot[mark=triangle*,color=brown] 
			coordinates {(3, 0.0270) (30, 0.1902) (60, 0.6034) (90, 1.8253) (120, 6.6797) (150, 13.3815) };
		\addplot 
			coordinates {(3, 1.0156) (30, 0.9504) (60, 1.9608) (90, 4.9581) (120, 11.8041) (150, 33.0246) }; 
		\end{axis}
	\end{tikzpicture}}
	\\
	\resizebox{0.26\textwidth}{!}{
	\begin{tikzpicture}	
		\begin{axis}[
	    	xlabel=Number of sensors $p$,
	    	ylabel=$\norm{x^* - x}/\norm{x^*}$,
	    	xtick={3,30,60,90,120,150},
		xmin=3,
		xmax=150,		
		symbolic x coords={3,30,60,90,120,150},		
	    	legend style={at={(0.5,-0.2)},anchor=north,legend columns=-1},
	    	ymajorgrids=true,
	    	grid style=dashed,
	]
		\addplot 
			coordinates {(3, 0.0483781047233582e-08) (30, 1.21403861608205e-15) (60, 8.60223172881467e-16) (90, 1.14182955103158e-15) (120, 1.26689021883866e-15) (150, 2.59537411335808e-15) };
		\addplot 
			coordinates {(3, 0.00735062351880407) (30, 0.00022816214771488813) (60, 2.288939887939634e-05) (90, 1.1770280567728685e-05) (120, 1.3011085399035208e-05) (150, 7.868698983610668e-06) };
		\addplot[mark=triangle*,color=brown] 
			coordinates {(3, 0.0850225969504879) (30, 0.0251441086194822) (60, 0.084687234990669) (90, 0.0896474079635188) (120, 0.00589022630290299) (150, 0.0370787356765995) };
		\addplot 
			coordinates {(3, 0.0483782026696649) (30, 2.23231491487202e-09) (60, 1.40871834682034e-08) (90, 3.44212824692906e-06) (120, 4.54680892142665e-09) (150, 2.89295827743251e-08) }; 
		\legend{SMT,MIQP,ETPG,$l_r/l_1$}
		\end{axis}
	\end{tikzpicture}}
\end{tabular}	
}
\end{tabular}
}
\caption{Simulation results showing number of iterations, execution time, and estimation error with respect to number of states and number of sensors.}
\label{fig:certificate1}
\end{figure*}

\subsection{Attacking an Unmanned Ground Vehicle}

We apply our algorithms to the model of a UGV, as detailed in~\cite{YasserETPGarXiv, Pajic_ICCPS}, under different types of sensor attacks.  We assume that the UGV moves along straight lines and completely stops before
rotating. Under these assumptions, we can describe the dynamics of the UGV as:
\begin{align*}
	\matrix{\dot{x} \\ \dot{v} } &= \matrix{0 & 1  \\ 0 & \frac{-B}{M} } \matrix{x
	\\ v  } + \matrix{0  \\ \frac{1}{M}  } F,
\end{align*}
where $x$ and $v$ are the states, corresponding to the UGV position and
linear velocity, respectively. The parameters $M$ and $B$  denote the mechanical mass and the translational 
friction coefficient. The inputs to 
the UGV is the force $F$.  The UGV is equipped with a GPS sensor which measures its position and two motor
encoders which measure the translational velocity. The resulting output equation is:
\begin{align*}
y = \matrix{1 & 0  \\ 0 & 1  \\ 0 & 1} \matrix{x \\ v } + \matrix{\psi_1 \\ \psi_2 \\ \psi_3},
\end{align*}
where $\psi_i$ is the measurement noise on the $i$th sensor which is assumed to be bounded. In our experiments, we used $M=0.8$~kg, $B=1$, $|{\psi_1}|^2 = 0.2$~m$^2$,  $|{\psi_2}|^2 = |{\psi_3}|^2 = 0.2$~(m/s)$^2$. 
 
The model is discretized with a time step equal to $0.1$~s. The SMT-based detector uses the discretized model along with sensor measurements to provide an estimate for the state vector, which is then used by a feedback controller to regulate the robot and follow a squared-shape path of length equal to $5$~m.

Figure \ref{fig:performance_tank} shows the performance of the SMT-based detector. The attacker alternates between corrupting the left and the right encoder measurements as shown in Figure \ref{fig:tank_attack_smt}. Three different types of attacks are considered. First, the attacker
corrupts the sensor signal with random noise. The next attack consists of a step function followed by a ramp. Finally, a replay-attack is mounted by replaying the previously measured
UGV velocity. The estimated position and velocity are shown in Figure
\ref{fig:tank_state_smt}. We recall that the SMT-based detector is also able to return the indicator variable vector $b$, denoting which sensors are under attack. Figure \ref{fig:tank_attack_smt} shows both the attack and the corresponding indicator variables as returned by the SMT-based detector.  The proposed algorithm is able to estimate the state and the support of the attack also in the presence of noise.

\begin{figure*}[t] 
\centering 
\subfigure[Estimated position and velocity versus ground truth.]
{\label{fig:tank_state_smt}
\resizebox{0.44\textwidth}{!}{
\includegraphics{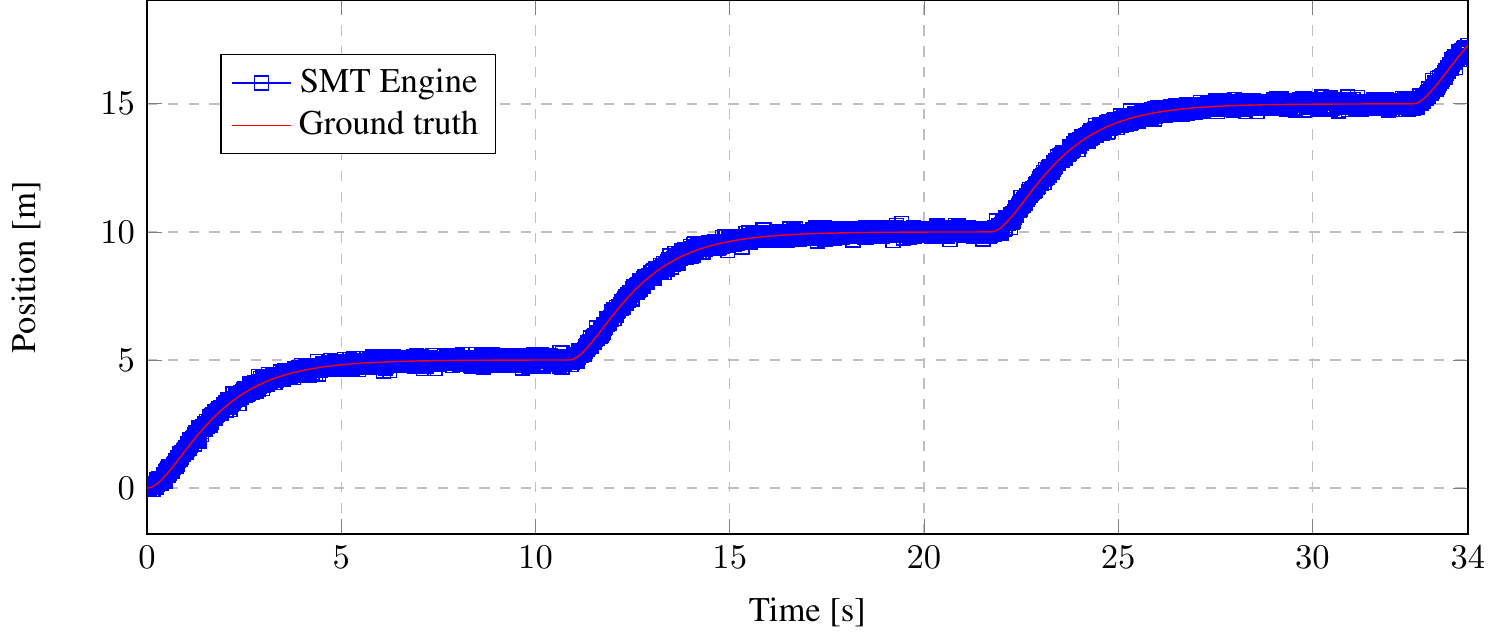}
		}
\resizebox{0.44\textwidth}{!}{
\includegraphics{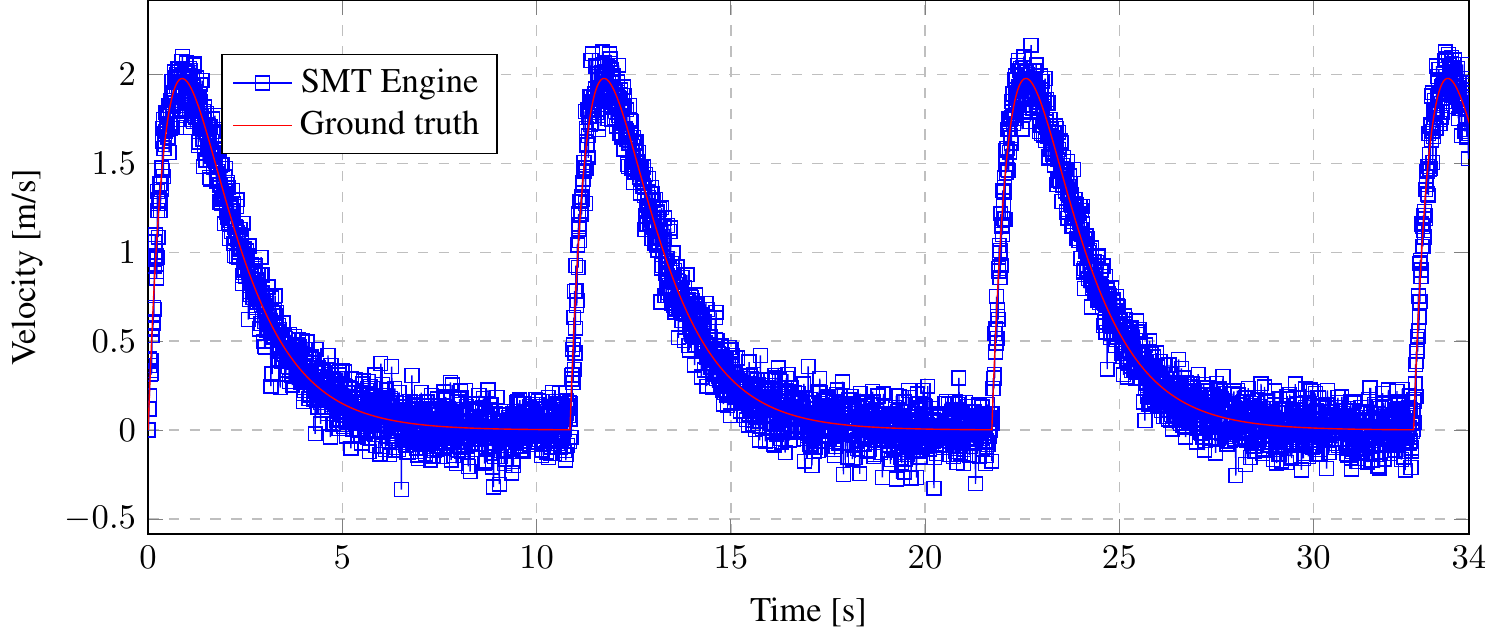}
		}
}\\
\subfigure[Attack signal on the left and right encoders. 
]{\label{fig:tank_attack_smt}
\resizebox{0.44\textwidth}{!}{
\includegraphics{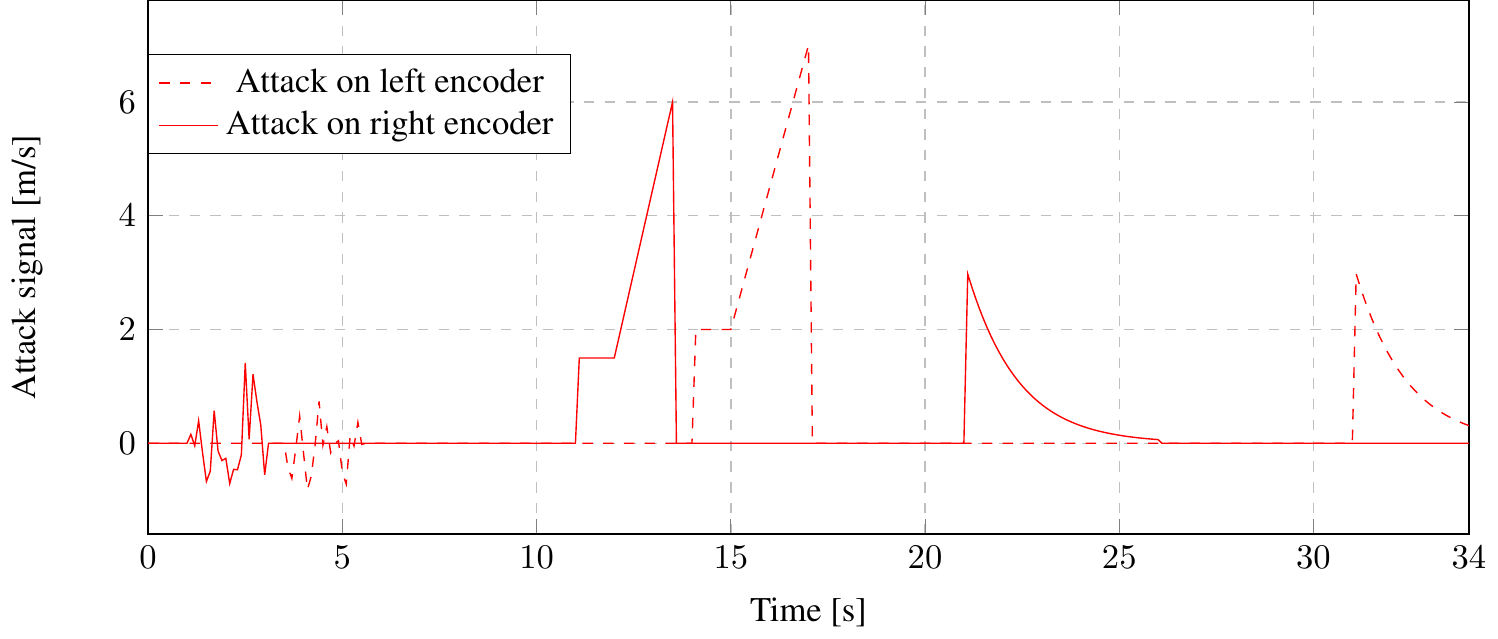}
		}
}
\subfigure[Indicator variables $b$ computed by the proposed SMT-based detector. 
]{		
\resizebox{0.44\textwidth}{!}{
\includegraphics{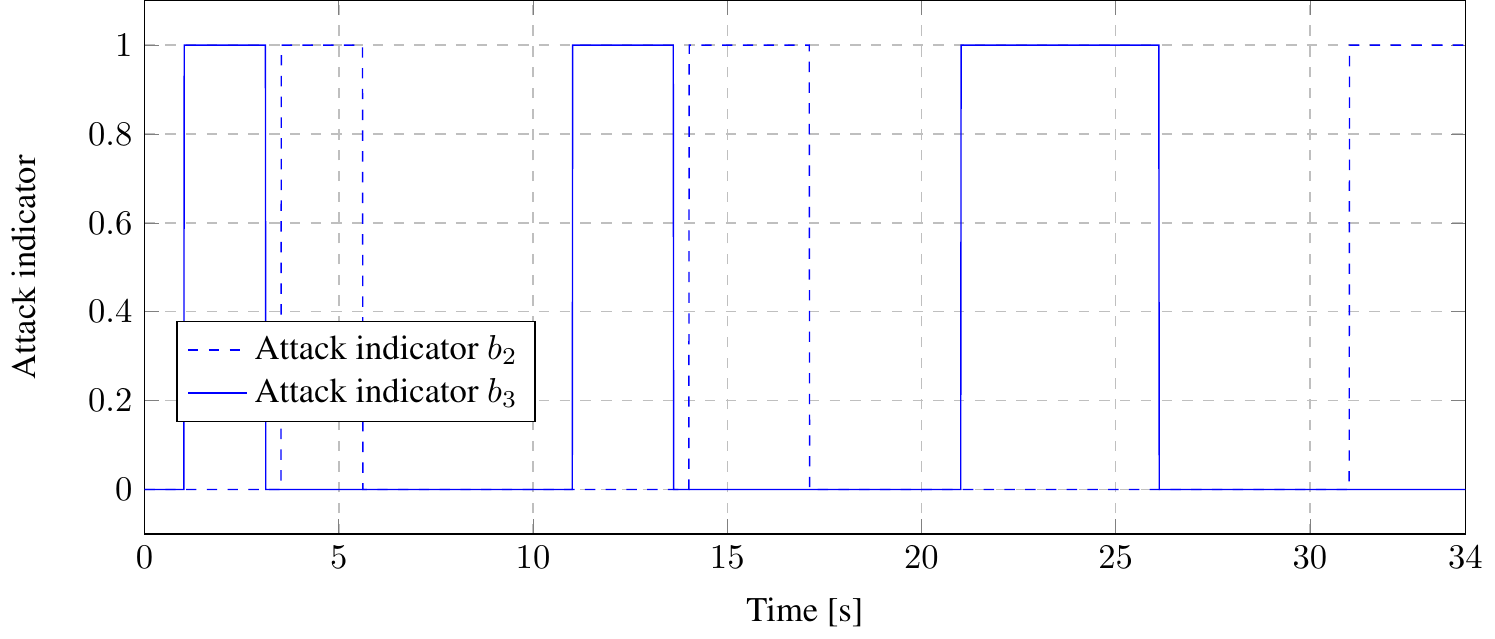}
		}
}
\caption{Performance of the UGV controller in the case when no attack takes place versus the case when the attack signal is applied to the UGV encoders. The objective is to move $5$ m, stop and perform a $90^o$ rotation, and repeat this pattern to follow a square path. The controller uses the proposed SMT-based approach to estimate the UGV states. In both cases we show the linear position and linear velocity (top), and the attack signal and its estimate (bottom).}
\label{fig:performance_tank}
\end{figure*}

\section{Conclusions}
\label{sec:conclusion}
We proposed a sound and complete algorithm which adopts the Satisfiability-Modulo-Theories paradigm to tackle the intrinsic combinatorial complexity of the secure state estimation problem for linear dynamical systems under sensor attacks. At the heart of our detector lie a set of routines that exploit the geometric structure of the problem to efficiently reason about inconsistency of sensor measurements and enhance the runtime performance. Our approach was  validated via numerical simulations, and demonstrated on an unmanned ground vehicle control problem. Future directions include the extension and the characterization of the proposed algorithm for nonlinear and hybrid dynamical systems.

\bibliographystyle{IEEEtran} 
\bibliography{bibliography2}

\clearpage
\appendix

\textbf{Fact 1:} For any two square matrices $A$ and $B$, both $AB$ and $BA$ have the same eigenvalues.

\textbf{Fact 2:} If $I - A$ is a positive definite matrix, then all eigenvalues of $A$ are strictly less than 1.

\begin{proposition}
Given a positive semidefinite matrix $A$ and a positive definite matrix $B$ of the same dimension, the following holds:
$$ \lambda \{A ( A+ B )^{-1}\} < 1$$
\label{prop:helper}
\end{proposition}

\begin{proof}
It follows from the positive (semi)definiteness assumptions of $A$ and $B$  that $(A+B)^{-1}$ is positive definite matrix and hence can be written using its square root matrix as:
$$ (A+B)^{-1} = (A+B)^{-\half} (A+B)^{-\half}.$$
Now, it follows from Fact 1 that $A ( A+ B )^{-1}$ have the same eigenvalues of $(A+B)^{-\half}  A (A+B)^{-\half}$. Now we have,
\begin{align*}
I - (A+B)^{-\half} A (A+B)^{-\half} 
&= (A+B)^{-\half}  (A + B) (A+B)^{-\half} \\&- (A+B)^{-\half}  A (A+B)^{-\half} \\
&= (A+B)^{-\half}  B (A+B)^{-\half}
\end{align*}
which is still positive definite. Hence, it follows from Fact 2 that all eigenvalues of $(A+B)^{-\half}  A (A+B)^{-\half}$ are strictly less than 1 and so are the eigenvalues of $A ( A+ B )^{-1}$.
\end{proof}

%

\end{document}